%% file: main.tex
\documentclass[10pt]{article}
\usepackage{amsmath,amssymb,amsthm,amsfonts}
\usepackage{enumerate,times,epsfig,color,mathdots,float}
\usepackage[utf8]{inputenc}
\usepackage{graphicx}
\usepackage{tikz}
\usepackage{kbordermatrix}
\usepackage{blkarray}

\usepackage[letterpaper,margin=1.1in]{geometry}
\usepackage[colorlinks=true,citecolor=blue]{hyperref}
\usepackage[numbers,sort&compress,square,comma]{natbib}

\usepackage{cleveref}
\usepackage{wrapfig}

\renewcommand{\arraystretch}{0.8}

\newtheorem{theorem}{Theorem}[section]
\newtheorem{CO}[theorem]{Corollary}
\newtheorem{LE}[theorem]{Lemma}

\newtheorem{RE}[theorem]{Remark}

\newtheorem*{EG}{Example}
\newtheorem{DE}[theorem]{Definition}
\newcounter{claim_nb}[theorem]
\newtheorem{claim}[claim_nb]{Claim}
\newtheorem*{claim*}{Claim}

\newcommand{\ignore}[1]{}
\newcommand{\rank}{\mathrm{rank}}

\newcommand{\supp}{\mathrm{support}}
\newcommand{\cuboid}{\mathrm{cuboid}}
\newcommand{\ind}{\mathrm{local}}
\newcommand{\local}{\mathrm{local}}

\newcommand{\mult}{\mathrm{mult}}
\newcommand{\mat}{\mathrm{Matroid}}

\newenvironment{cproof}
{\begin{proof}
		[Proof of Claim.]
		\vspace{-1.2\parsep}}
	{ \end{proof}}

\usepackage{bigfoot}

\DeclareNewFootnote{AAffil}[arabic]
\DeclareNewFootnote{ANote}[fnsymbol]

\usepackage{etoolbox}
\makeatletter
\patchcmd\maketitle{\def\@makefnmark{\rlap{\@textsuperscript{\normalfont\@thefnmark}}}}{}{}{}
\makeatother

\makeatletter
\def\thanksAAffil#1{%
	\footnotemarkAAffil\protected@xdef\@thanks{\@thanks%
		\protect\footnotetextAAffil[\the \c@footnoteAAffil]{#1}}%
}
\def\thanksANote#1{%
	\footnotemarkANote%
	\protected@xdef\@thanks{\@thanks%
		\protect\footnotetextANote[\the \c@footnoteANote]{#1}}%
}
\makeatother

\title{From coordinate subspaces over finite fields to ideal multipartite uniform clutters}

\author{Ahmad Abdi
\thanksAAffil{Department of Mathematics,
	London School of Economics and Political Science, London WC2A 2AE,
	UK, \url{a.abdi1@lse.ac.uk }}
\and
Dabeen Lee
\thanksAAffil{Department of Industrial and Systems Engineering, KAIST, Daejeon 34126, Republic of Korea, \url{dabeenl@kaist.ac.kr}}$~$
\thanksANote{Corresponding author}}

\begin{document}
	
	\maketitle
	
	\begin{abstract}
	Take a prime power $q$, an integer $n\geq 2$, and a coordinate subspace $S\subseteq GF(q)^n$ over the Galois field $GF(q)$. One can associate with $S$ an $n$-partite $n$-uniform clutter $\mathcal{C}$, where every part has size $q$ and there is a bijection between the vectors in $S$ and the members of $\mathcal{C}$. 
	
	In this paper, we determine when the clutter $\mathcal{C}$ is \emph{ideal}, a property developed in connection to Packing and Covering problems in the areas of Integer Programming and Combinatorial Optimization. Interestingly, the characterization differs depending on whether $q$ is $2,4$, a higher power of $2$, or otherwise. Each characterization uses crucially that idealness is a \emph{minor-closed property}: first the list of excluded minors is identified, and only then is the global structure determined. A key insight is that idealness of $\mathcal{C}$ depends solely on the underlying matroid of $S$. 
	
	Our theorems also extend from idealness to the stronger \emph{max-flow min-cut} property. As a consequence, we prove the Replication and $\tau=2$ Conjectures for this class of clutters.\\
	
	\end{abstract}

\input{introduction}

\input{multipartite}

\input{matroids}

\input{mfmc-odd}

\input{p=2-structure}

\input{p=2-At}

\input{q=4}

\input{q_4}

\input{replication}

	\section*{Acknowledgements and Funding}
	We would like to thank G\'erard Cornu\'ejols for helpful discussions. This research is supported, in part, by KAIST Starting Fund (KAIST-G04220016), FOUR Brain Korea 21 Program (NRF-5199990113928), and National Research Foundation of Korea (NRF-2022M3J6A1063021).

\appendix

\input{appendix}

\end{document}

%% file: introduction.tex
\section{Introduction}\label{sec:intro}

Let $V$ be a finite set of \emph{elements}, and let $\mathcal{C}$ be a family of subsets of $V$ called \emph{members}. A \emph{cover} is defined as a subset of $V$ that intersects every member in $\mathcal{C}$. Given weights $w\in\mathbb{Z}_+^V$, a minimum weight cover can be computed by solving the integer program $\min\{w^\top x:M(\mathcal{C})x\geq\mathbf{1},x\in\mathbb{Z}_+^V\}$,
where $M(\mathcal{C})$ is the incidence matrix of $\mathcal{C}$ whose columns are labeled by the elements and whose rows are the incidence vectors of the members. The linear programming relaxation of this integer program is the problem of minimizing $w^\top x$ over the \emph{associated set covering polyhedron} given by $Q(\mathcal{C}):=\{x\in\mathbb{R}^V:M(\mathcal{C})x\geq\mathbf{1},x\geq\mathbf{0}\}.$
For the purpose of finding a minimum weight cover, we may assume without loss of generality that no member properly contains another, in which case we call $\mathcal{C}$ a \emph{clutter} over ground set $V$~\cite{Edmonds70}. A necessary and sufficient condition for the relaxation to return an integer solution for any $w\in\mathbb{Z}_+^V$, thereby giving a minimum weight cover, is that every extreme point of $Q(\mathcal{C})$ is an integral vector, in which case we say that $\mathcal{C}$ is \emph{ideal}~\cite{Cornuejols94}. 

Every clutter whose members are pairwise disjoint is obviously ideal. Many non-trivial examples of ideal clutters can be found in Combinatorial Optimization -- let us mention a few here: the clutter of $st$-paths of a graph~\cite{Menger27}, (inclusionwise) minimal $st$-cuts of a graph~\cite{Duffin62}, minimal $T$-joins of a graph~\cite{Edmonds73}, minimal $T$-cuts of a graph~\cite{Edmonds73}, and odd circuits of a signed graph that has no odd-$K_5$ minor~\cite{Guenin01}. Each of these examples has as ground set the edge set of the associated graph. In general, it is co-NP-complete to decide whether a clutter is ideal~\cite{Ding08}, and understanding the various aspects of the theory of ideal clutters is one of the long-standing open research directions in the area: 11 of the 18 conjectures in the book \emph{Combinatorial Optimization.\ Packing and Covering}~\cite{Cornuejols01} are directly about general or special instances of ideal clutters.

\paragraph{Multipartite uniform clutters.}
In this paper we introduce a novel approach to discover and understand ideal clutters, by studying the notion of \emph{multipartite uniform clutters} defined as follows. A multipartite uniform clutter $\mathcal{C}$ is obtained as a family of hyperedges of an \emph{$n$-partite} hypergraph whose vertices are partitioned into $n$ nonempty disjoint subsets $V_1,\ldots, V_n$ for some $n\geq 1$, and every hyperedge intersects each of the subsets exactly once. Then all members of $\mathcal{C}$ have an equal size $n$, and therefore, $\mathcal{C}$ is \emph{uniform} and a \emph{clutter}. In particular, in a multipartite uniform clutter, the size of a member is equal to the number of partitions.
For example, $Q_6$, the clutter of triangles in $K_4$ given by 
$Q_6:=\left\{\{1,3,5\},\{1,4,6\},\{2,3,6\},\{2,4,5\}\right\}$,
is a 3-partite 3-uniform clutter over ground set $\{1,\ldots,6\}$ partitioned into $\{1,2\}\cup \{3,4\}\cup \{5,6\}$. The class of multipartite uniform clutters looks restricted, but in fact, it is general enough to understand the entire class of ideal clutters. More precisely, it was shown in~\cite{Abdi-cuboid} that if we had a characterization of when a multipartite uniform clutter is ideal, then this would in turn completely characterize ideal clutters. This is because any given clutter can be ``locally embedded" in a multipartite uniform clutter~\cite{Abdi-cuboid}, and we will discuss related ideas in \Cref{sec:cuboid}. This connection allows us to take a different angle on understanding idealness by studying multipartite uniform clutters.

\paragraph{Vector spaces over $GF(q)$.} Thanks to their special structure, one may take advantage of a geometric framework for constructing multipartite uniform clutters. To explain it, take a prime power $q$ and $GF(q)$, the \emph{Galois field of order $q$}. For convention, we denote by $0$ and $1$ the additive and multiplicative identities of $GF(q)$, respectively. When $q$ is a power of a prime number $p$, we call $p$ the \emph{characteristic} of $GF(q)$. $GF(q)^n$ for some $n\geq 1$ is the set of $n$-dimensional vectors whose coordinates are in $GF(q)$ and is called a \emph{coordinate space}. We say that any vector subspace of the coordinate space over $GF(q)$ is a \emph{coordinate subspace}. Throughout the paper, we refer to a coordinate subspace over $GF(q)$ as a \emph{vector space over $GF(q)$} or simply as a coordinate subspace. For any vector space $S\subseteq GF(q)^n$ over $GF(q)$, there exists a matrix $A$ whose entries are in $GF(q)$ such that $S=\left\{x\in GF(q)^n:Ax=\mathbf{0}\right\}$
where $\mathbf{0}$ denotes the vector of all zeros of appropriate dimension and all equalities in the system $Ax=\mathbf{0}$ are over $GF(q)$. Given the coordinate subspace $S$, we construct a multipartite uniform clutter in the following way. Taking $n$ disjoint copies $V_1,\ldots, V_n$ of $GF(q)$, we can view $GF(q)^n$ as $V_1\times\cdots\times V_n$ so that $S$ is a subset of $V_1\times\cdots\times V_n$. The {\it multipartite uniform clutter of $S$} is the clutter over ground set $V_1\cup\cdots\cup V_n$ defined by
\begin{equation*}
\mult(S):=\left\{\{x_1,\ldots,x_n\}:~(x_1,\ldots,x_n)\in S ,~x_i\in V_i~~\text{for}~i\in[n] \right\}.
\end{equation*}
Here, the size of a member equals the number of partitions $n$, and $\mult(S)$ is an $n$-partite $n$-uniform clutter. For example, $R_{1,1}:=\{(0,0,0),(0,1,1),(1,0,1),(1,1,0)\}$ is a vector space over $GF(2)$, and $R_{1,1}$ is equivalent to $\{(1,3,5),(1,4,6),(2,3,6),(2,4,5)\}\subseteq \{1,2\}\times\{3,4\}\times\{5,6\}$. So, $\mult(R_{1,1})$ is isomorphic\footnote{Given clutters $\mathcal{C},\mathcal{C}^\prime$, we say that $\mathcal{C}$ is \emph{isomorphic} to $\mathcal{C}^\prime$ and write $\mathcal{C}\cong\mathcal{C}^\prime$ if $\mathcal{C}^\prime$ can be obtained from $\mathcal{C}$ after relabeling the elements of $\mathcal{C}$.} to $Q_6$. There is a one-to-one correspondence between the members of $\mult(S)$ and the vectors in $S$. Although we focus on vector spaces over a finite field, we remark that the definition of multipartite uniform clutters extends to any subset of the direct product of finite groups. We discuss this further in \Cref{sec:cuboid}. 

\paragraph{Binary spaces.} Abdi, Cornu\'ejols, Guri\u{c}anov\'{a}, and Lee~\cite{Abdi-cuboid} considered vector spaces over $GF(2)$, often referred to as \emph{binary} spaces, and provided a characterization of when their multipartite uniform clutters are ideal. For example, $\mult(R_{1,1})=Q_6$ is ideal~\cite{Seymour77}. The characterization is in terms of {\it clutter minors}, or simply {\it minors}. Given a clutter $\mathcal{C}$ over ground set $V$ and disjoint subsets $I,J$ of $V$, we define $\mathcal{C}\setminus I/J$ as the clutter over $V-(I\cup J)$ that consists of the minimal sets of $\left\{C-J:~C\in\mathcal{C},~C\cap I=\emptyset\right\}$, and we say that $\mathcal{C}\setminus I/J$ is {\it the minor of $\mathcal{C}$} obtained after {\it deleting $I$} and {\it contracting $J$}. We call it a {\it proper} minor if $I\cup J\neq\emptyset$. It is well-known that if a clutter is ideal, then so is every minor~\cite{Seymour77}. It was proved in \cite{Abdi-cuboid} that for a vector space $S$ over $GF(2)$, $\mult(S)$ is ideal if and only if $\mult(S)$ has none of three special clutters as a minor if and only if the binary matroid corresponding to $S$ has the so-called sums of circuits property.

\paragraph{Our results I.} 
Motivated by the result of \cite{Abdi-cuboid} mentioned above, given a vector space $S$ over an arbitrary finite field $GF(q)$, when is $\mult(S)$ is ideal? 
In this paper, we completely answer this question. We divide our analysis into three cases. First, we consider prime powers that are odd, secondly the $q=4$ case, and thirdly powers of 2 greater than 4. What follows is a summary of our main results for the three cases.

For our first result, we need two more definitions. The \emph{support} of a vector $x\in GF(q)^n$ is defined as $\supp(x):=\{i\in[n]:x_i\neq 0\}$. Moreover, denote by $\Delta_3$ the clutter over ground set $\{1,2,3\}$ whose members are $\{1,2\},\{2,3\},\{3,1\}$. Notice that $\Delta_3$ is the clutter of edges in a triangle and that $\Delta_3$ is non-ideal because $\left(\frac{1}{2},\frac{1}{2},\frac{1}{2}\right)$ is a fractional extreme point of the associated set covering polyhedron $Q(\Delta_3)$.

\begin{theorem}[proved in \Cref{sec:mfmc-odd}]\label{q odd}
Take an odd prime power $q$, and let $S$ be a vector space over $GF(q)$. Then the following statements are equivalent:
\begin{enumerate}[(i)]
\item $\mult(S)$ is ideal,
\item $S$ admits a basis with vectors of pairwise disjoint supports, 
\item $\mult(S)$ contains no $\Delta_3$ as a minor.
\end{enumerate}
\end{theorem}

The case of $GF(4)$ allows more general structures in the vector space. We say that row vectors $v^1,\ldots, v^r$ with $r\geq 2$ form a \emph{sunflower} if, after permuting the coordinates, the vectors are of the form
\[
\begin{matrix}
	v^{1}\\
	v^{2}\\
	\vdots\\
	v^{r}
\end{matrix}
\left[
\begin{array}{c|c|c|c|c}
	u^0 & u^1 & \mathbf{0} & \cdots & \mathbf{0} \\\
	u^0 & \mathbf{0} & u^2 & \cdots & \mathbf{0}\\
	\vdots    & \vdots    &   \vdots  & \ddots & \vdots \\
	u^0 & \mathbf{0} & \mathbf{0} & \cdots & u^r
\end{array}
\right]
\]
where $u^0,u^1\ldots, u^r$ are some row vectors with nonzero entries and $\mathbf{0}$ denotes a row vector of all zeros of appropriate length.

\begin{theorem}[proved in \Cref{sec:q=4}]\label{q=4}
	Let $S$ be a vector space over $GF(4)$. Then the following statements are equivalent:
	\begin{enumerate}[(i)]
		\item $\mult(S)$ is ideal,
		\item $S=S_1\times\cdots\times S_k$ where each $S_i$ has dimension at most $1$ or admits a sunflower basis,
		\item $\mult(S)$ contains no $\Delta_3$ as a minor.
	\end{enumerate}
\end{theorem}

Lastly, for the case when $q$ is a power of 2 greater than 4, we define another small non-ideal clutter. $C_5^2$ is the clutter over ground set $\{1,\ldots,5\}$ whose members are $\{1,2\},\{2,3\},\{3,4\},\{4,5\},\{5,1\}$. $C_5^2$ is the clutter of edges in a cycle of length 5, and notice that $C_5^2$ is non-ideal because $\left(\frac{1}{2},\frac{1}{2},\frac{1}{2},\frac{1}{2},\frac{1}{2}\right)$ is a fractional extreme point of the associated polyhedron $Q(C_5^2)$. 

\begin{theorem}[proved in \Cref{sec:q>4}]\label{q>4}
	Let $q$ be a power of 2 such that $q>4$, and let $S$ be a vector space over $GF(q)$. Then the following statements are equivalent:
	\begin{enumerate}[(i)]
		\item $\mult(S)$ is ideal,
		\item $S$ admits a basis with vectors of pairwise disjoint supports, 
		\item $\mult(S)$ contains no $C_5^2$ as a minor.
	\end{enumerate}
\end{theorem}

\Cref{q odd}, \Cref{q=4}, and \Cref{q>4} lead to the conclusion that when $q$ is a prime power other than  2, the class of coordinate subspaces whose multipartite uniform clutter is ideal has restricted structures. Nevertheless, the main takeaway of this paper is that we propose a novel framework to study and generate idealness by multipartite uniform clutters and complete the analysis of the natural class of multipartite uniform clutters obtained from coordinate subspaces. Our analysis is based on sophisticated interplays between clutters and underlying matroids.

\paragraph{Our results II.}
We take one step further to understand the \emph{max-flow min-cut (MFMC) property}~\cite{Seymour77} for the multipartite uniform clutters from coordinate subspaces. While the idealness of a clutter corresponds to the integrality of the associated set covering polyhedron, the MFMC property is the analogue of \emph{total dual integrality}~\cite{Edmonds77,Hoffman74}. To formalize this, given a clutter $\mathcal{C}$ over ground set $V$ with weights $w\in\mathbb{Z}_+^V$, we consider $\tau(\mathcal{C},w):=\min\{w^\top x:M(\mathcal{C})x\geq\mathbf{1},x\in\mathbb{Z}_+^V\}$ and $\nu(\mathcal{C},w):=\max\left\{\mathbf{1}^\top y:M(\mathcal{C})^\top y\leq w,y\in\mathbb{Z}_+^\mathcal{C}\right\}$. Note that $\tau(\mathcal{C},w)$ computes the minimum weight of a cover of $\mathcal{C}$, whereas $\nu(\mathcal{C},w)$ computes the maximum size of a \emph{packing} of members of $\mathcal{C}$ such that each element $v$ appears in at most $w_v$ members in the packing. Here, we say that $\mathcal{C}$ has the MFMC property if $\tau(\mathcal{C},w)=\nu(\mathcal{C},w)$ holds for every $w\in\mathbb{Z}_+^V$. Hence, the MFMC property of $\mathcal{C}$ is equivalent to the total dual integrality of the linear system $M(\mathcal{C})x\geq\mathbf{1},x\geq\mathbf{0}$, and therefore it follows that the MFMC property implies idealness. The following result provides a complete characterization of the MFMC property for the multipartite uniform clutters from vector spaces.
\begin{theorem}[proved in \Cref{sec:mfmc-odd}]\label{mfmc}
	Take any prime power $q$, and let $S$ be a vector space over $GF(q)$. Then the following statements are equivalent:
	\begin{enumerate}[(i)]
		\item $\mult(S)$ has the max-flow min-cut property,
		\item $S$ admits a basis with vectors of pairwise disjoint supports, 
		\item $\mult(S)$ has none of $\Delta_3, Q_6$ as a minor.
	\end{enumerate}
\end{theorem}

As a corollary, idealness and the MFMC property coincide when $q$ is an odd prime power or $q$ is a power of 2 greater than 4. In contrast, there is an example of a vector space over $GF(4)$ whose multipartite uniform clutter is ideal but does not have the MFMC property. We demonstrate this example in \Cref{sec:replication}. \Cref{mfmc} also has a consequence on the \emph{Replication Conjecture}, proposed by Conforti and Cornu\'ejols~\cite{Conforti93}. In particular, the Replication Conjecture is a set covering analogue of the Duplication Lemma for perfect graphs~\cite{Lovasz72}.

\begin{CO}[proved in \Cref{sec:replication}]\label{replication}
The Replication Conjecture holds true for the class of multipartite uniform clutters from coordinate subspaces.
\end{CO}

Another corollary of \Cref{mfmc} is on the \emph{$\tau=2$ Conjecture}, proposed by Cornu\'ejols, Guenin, and Margot~\cite{Cornuejols00}. They showed that if the $\tau=2$ Conjecture holds, then so does the Replication Conjecture~\cite{Cornuejols00}, providing a way of tackling the Replication Conjecture.
\begin{CO}[proved in \Cref{sec:replication}]\label{tau=2}
	The $\tau=2$ Conjecture holds true for the class of multipartite uniform clutters from coordinate subspaces.
\end{CO}

 We will formally state the Replication Conjecture and the $\tau=2$ Conjecture along with the proofs of \Cref{replication} and \Cref{tau=2} in \Cref{sec:replication}.

\paragraph{Summary and organizations of the paper.} This paper provides a complete characterization of when the multipartite uniform clutter of a coordinate subspace is ideal and when it has the MFMC property. The proofs of our main results are based on applications of the theory of ideal clutters and matroid theory. Tools from ideal clutters and matroid theory are presented in \Cref{sec:cuboid} and \Cref{sec:matroid}, respectively.

The first result we prove in this paper is \Cref{mfmc} which characterizes the MFMC property of the multipartite uniform clutter of a vector space over $GF(q)$ for any prime power $q$. In fact, \Cref{q odd} for the idealness under an odd prime power $q$ shares much of the proof with \Cref{mfmc}. Hence, we prove the two theorems in \Cref{sec:mfmc-odd}.

For the idealness under the case of powers of 2, we need more techniques. In \Cref{sec:p=2-structure}, we provide some properties of the underlying matroid of a vector space over $GF(2^k)$ for $k\geq 2$. In \Cref{sec:p=2-At}, we develop some tools for understanding vector spaces of a certain structure that appear for the case of powers of 2. We divide our analysis of the case of powers of 2 into the $q=4$ case and the case of $q=2^k$ for $k\geq 3$. The $q=4$ case, \Cref{q=4}, is covered in \Cref{sec:q=4}. The other case, \Cref{q>4}, is presented in \Cref{sec:q>4}.

We conclude the paper by proving~\Cref{replication} and~\Cref{tau=2} on the Replication Conjecture and the $\tau=2$ Conjecture, respectively, for the class of multipartite clutters from coordinate subspaces in~\Cref{sec:replication}.

%% file: multipartite.tex
\section{Multipartite uniform clutters}\label{sec:cuboid}

In this section, we develop some useful tools for understanding when the multipartite uniform clutter of a vector space of a finite field is ideal. Let $V_1,\ldots,V_n$ be $n$ nonempty sets, and take a subset $S$ of $V_1\times\cdots\times V_n$. We would take $V_i=GF(q)$ for $i\in[n]$ for a vector space over $GF(q)$, but we may take arbitrary finite sets that do not necessarily have the same size. Then the multipartite uniform clutter of $S$, denoted $\mult(S)$, is defined as the clutter over ground set $V_1\cup\cdots\cup V_n$ whose members are $\{x_1,\ldots,x_n\}$ for $(x_1,\ldots,x_n)\in S$.
Here, $S$ need not be a vector space. When each $V_i$ has size two, $\mult(S)$ for $S\subseteq V_1\times\cdots\times V_n$ coincides with the {\it cuboid of $S$}, denoted $\cuboid(S)$~\cite{Abdi-idealmnp,Abdi-cuboid}. In that case, $V_1\times\cdots\times V_n$ is isomorphic to $\{0,1\}^n$, so cuboids correspond to vertex subsets of the $n$-dimensional 0,1 hypercube, and this is how the name cuboid is coined. Therefore, for a binary space $S$, we have that $\mult(S)=\cuboid(S)$. 

\begin{RE}\label{mult-equiv}
	Let $\mathcal{C}$ be a clutter, and let $V_1,\ldots,V_n$ be $n$ non-empty sets. Then the following statements are equivalent:
	\begin{enumerate}[(i)]
		\item $\mathcal{C}$ is isomorphic to $\mult(S)$ for some $S\subseteq V_1\times\cdots\times V_n$,
		\item the ground set of $\mathcal{C}$ can be partitioned into $V_1,\ldots,V_n$ so that for every $C\in\mathcal{C}$, $|C\cap V_i|=1$ for all $i\in[n]$.
	\end{enumerate}
\end{RE}

\Cref{mult-equiv} provides a different yet equivalent definition of multipartite uniform clutters. Now that we have seen \Cref{mult-equiv}, we know that the incidence matrix of a multipartite uniform clutter can be partitioned. To be more precise, notice that if a multipartite uniform clutter's ground set is partitioned into $n$ non-empty parts $V_1,\ldots,V_n$, then the columns of the member-element incidence matrix $M(\mathcal{C})$ of $\mathcal{C}$ can be partitioned into $n$ groups, corresponding to $V_1,\ldots, V_n$, so that a row has precisely one nonzero entry in each group. For instance,
$$M(Q_6)=\kbordermatrix{&0& 1 && 0  & 1 && 0 & 1 \\
	(0,0,0)&1& &\vrule&1& &\vrule&1& \\ 
	(0,1,1)&1& &\vrule& &1&\vrule& &1\\
	(1,0,1)& &1&\vrule&1& &\vrule& &1\\
	(1,1,0)& &1&\vrule& &1&\vrule&1&
}.$$
As mentioned in \Cref{sec:intro}, one can also view a multipartite uniform clutter with parts $V_1,\ldots, V_n$ as the clutter of hyperedges of an $n$-partite $n$-uniform hypergraph whose vertex set is partitioned into $V_1\cup\cdots\cup V_n$.

\paragraph{Isomorphism.} 

\begin{RE}\label{RE:relabel-elts}
	Take an integer $n\geq 1$ and a prime power $q$, and let $S\subseteq GF(q)^n$ be a vector space over $GF(q)$. Let $f_i:GF(q)\rightarrow GF(q)$ be a bijection for $i\in[n]$, and $g:GF(q)^n\rightarrow GF(q)^n$ be the bijection defined as
	\[
	g(x):=\left(f_1(x_1),\ldots,f_n(x_n)\right),\quad x\in GF(q)^{n}.
	\]
	Then $S\cong g(S)$ and $\mult(S)\cong \mult\left(g(S)\right)$.
\end{RE}

\paragraph{Products of set systems and clutters.}
Take two integers $n_1,n_2\geq 1$. Let $V_1,\ldots, V_{n_1}$ be $n_1$ nonempty sets, and let $S_1$ be a subset of $V_1\times\cdots\times V_{n_1}$. Let $U_1,\ldots,U_{n_2}$ be $n_2$ nonempty sets, and let $S_2$ be a subset of $U_1\times\cdots \times U_{n_2}$. Recall that the product of $S_1$ and $S_2$ is defined as
$
S_1\times S_2=\left\{(x_1,x_2):~x_1\in S_1,~x_2\in S_2\right\}.
$
We also define products of clutters. Let $\mathcal{C}_1,\mathcal{C}_2$ be two clutters over disjoint ground sets $E_1,E_2$. The {\it product of $\mathcal{C}_1$ and $\mathcal{C}_2$}, denoted $\mathcal{C}_1\times\mathcal{C}_2$, is defined as the clutter over ground set $E_1\cup E_2$ whose members are
$
\mathcal{C}_1\times\mathcal{C}_2=\left\{C_1\cup C_2:~C_1\in \mathcal{C}_1,~C_2\in\mathcal{C}_2\right\}.
$
Having defined the product of two clutters, we define the product of two multipartite uniform clutters $\mult(S_1)$ and $\mult(S_2)$. In fact, we can show the following:
\begin{LE}\label{set-product}
The following statements hold: \begin{enumerate}
\item $\mult(S_1)\times\mult(S_2)=\mult(S_1\times S_2)$.
\item Let $\mathcal{C}_1,\mathcal{C}_2$ be clutters over disjoint ground sets. If $\mathcal{C}_1,\mathcal{C}_2$ have the idealness (resp. MFMC) property, then so does $\mathcal{C}_1\times\mathcal{C}_2$.
\item Take two integers $n_1,n_2\geq 1$. Let $V_1,\ldots, V_{n_1}$ be $n_1$ nonempty sets, and let $S_1$ be a subset of $V_1\times\cdots\times V_{n_1}$. Let $U_1,\ldots,U_{n_2}$ be $n_2$ nonempty sets, and let $S_2$ be a subset of $U_1\times\cdots \times U_{n_2}$. If $\mult(S_1),\mult(S_2)$ are ideal, then so is $\mult(S_1\times S_2)$.
\end{enumerate}
\end{LE}
\begin{proof}
{\bf (1)}
Let $C_1\in\mult(S_1)$ and $C_2\in\mult (S_2)$. Then $C_1=\left\{x_1,\ldots, x_{n_1}\right\}$ for some $x=(x_1,\ldots,x_{n_1})\in S_1$ and $C_2=\left\{y_1,\ldots, y_{n_2}\right\}$ for some $y=(y_1,\ldots,y_{n_2})\in S_2$. Moreover, $(x,y)\in S_1\times S_2$ and $C_1\cup C_2\in \mult(S_1\times S_2)$. Similarly, we can show that if $C\in\mult(S_1\times S_2)$, then $C=C_1\cup C_2$ for some $C_1\in\mult(S_1)$ and $C_2\in\mult (S_2)$. Therefore, we obtain $\mult(S_1)\times\mult(S_2)=\mult(S_1\times S_2)$.
{\bf (2)} is routine and an argument can be found in~(\cite{Abdi-cuboid}, \S5).
{\bf (3)} follows from (1) and (2).
\end{proof}

So, if a set can be represented as the product of some smaller sets, we can check if its multipartite uniform clutter is ideal by studying the smaller sets and their multipartite uniform clutters. In particular, we will use this proposition to show implication {\bf (iii)}$\to${\bf (i)} in \Cref{q odd},  \Cref{q=4}, and \Cref{q>4}.

\paragraph{Projection and restriction of set systems.} Take an integer $n\geq 1$. Let $V_1,\ldots, V_{n}$ be $n$ nonempty sets, and let $S$ be a subset of $V_1\times\cdots\times V_{n}$. Given $J\subseteq[n]$ and $x\in S$, $x/J$ denote the subvector of $x$ that consists of the coordinates not in $J$. For $J\subseteq[n]$, the set obtained from $S$ after {\it dropping} the coordinates in $J$ is $\left\{x/J:~x\in S\right\}$. We will refer to a set obtained from $S$ after dropping some coordinates as a {\it projection} of $S$. Next, let $U_i$ be a nonempty subset of $V_i$ for $i\in[n]$. Here, $U_i$ need not be a proper subset of $V_i$. Then the set obtained from $S$ after {\it restricting} to $U_1\times\cdots\times U_n$ is what is obtained from $S\cap \left(U_1\times\cdots\times U_n\right)$ after dropping the coordinates where the points in $S\cap \left(U_1\times\cdots\times U_n\right)$ agree on. We will refer to a set obtained from $S$ after restricting $S$ to some subset of $V_1\times\cdots\times V_n$ as a {\it restriction} of $S$.

\begin{LE}\label{projection}
Let $S^\prime$ be a set that is either a projection or a restriction of $S$. Then $\mult(S^\prime)$ is a minor of $\mult(S)$.
\end{LE}
\begin{proof}
Suppose first that $S^\prime$ is a projection, say for some $J\subseteq[n]$, $S^\prime$ is obtained from $S$ after dropping the coordinates of $J$. Then $\mult(S^\prime)$ is the minor of $\mult(S)$ obtained after contracting the elements in $V_j$ for $j\in J$.

Suppose next that $S^\prime$ is a restriction, say for some nonempty subset $U_1\times\cdots\times U_n$ of $V_1\times\cdots\times V_n$, $S^\prime$ is obtained after restricting $S$ to $U_1\times\cdots\times U_n$. Then $\mult(S^\prime)$ is the minor of $\mult(S)$ obtained after deleting the elements in $(V_i\setminus U_i)$ for $i\in[n]$ and contracting the elements in $V_j$ for $j\in J$ where $J$ is the set of coordinates where the points in $S\cap \left(U_1\times\cdots\times U_n\right)$ agree on.
\end{proof}

\paragraph{Localizations.
} We mentioned before that a clutter is ideal if and only if every minor of it is ideal. In this section, we will define and study  
\emph{localizations} that appear as a minor of a multipartite uniform clutter.

\begin{DE} Given a multipartite uniform clutter $\mathcal{C}$ whose ground set is partitioned into non-empty parts $V_1,\ldots, V_n$, 
	a \emph{localization of $\mathcal{C}$}  is any minor obtained from $\mathcal{C}$ after contracting precisely one element from each $V_i$. \end{DE}

Thus, a localization of $\mathcal{C}$ is obtained after contracting $v_1,\ldots, v_n$ for some $v=(v_1,\ldots,v_n)\in V_1\times\cdots\times V_n$. As $\mathcal{C}=\mult(S)$ for some $S\subseteq V_1\times\cdots\times V_n$ by \Cref{mult-equiv}, the localization is equal to
\begin{align*}
\ind(S,v)&:=\mult(S)/\{v_1,\ldots,v_n\}\\
&=\left\{\text{the minimal sets of }\left\{\{x_1,\ldots,x_n\}-\{v_1,\ldots,v_n\}:\;(x_1,\ldots,x_n)\in S\right\}\right\}.
\end{align*}
We call $\ind(S,v)$ the {\it localization of $\mult(S)$ with respect to $v$}. So, every localization of $\mathcal{C}$ is equal to $\ind(S,v)$ for some $v$ and that $\ind(S,v)=\{\emptyset\}$ if $v\in S$.  In \cite{Abdi-cuboid}, localizations of a cuboid are referred to as \emph{induced clutters}.

It turns out that a multipartite uniform clutter is ideal if and only if all localizations are ideal; let us prove this in the remainder of this section. We say that a clutter is {\it minimally non-ideal} if it is non-ideal but every proper minor of it is ideal. We need the following lemma.

\begin{LE}\label{mni_no_partitionable_cols}
	Let $\mathcal{C}$ be a minimally non-ideal clutter, and let $V$ denote the ground set of $\mathcal{C}$. Then there is no subset $U$ of $V$ satisfying $|C\cap U|=1$ for every member $C$ of $\mathcal{C}$.
\end{LE}
\begin{proof}
	Since $\mathcal{C}$ is non-ideal, $P(\mathcal{C})=\{{\bf 1}\geq x\geq {\bf 0}:M(\mathcal{C})x\geq{\bf 1}\}$ has a fractional extreme point $x^*$. Let $v\in V$. Notice that $P(\mathcal{C}/v)$ and $P(\mathcal{C}\setminus v)$ are obtained from $P(\mathcal{C})\cap\{x:x_v=0\}$ and $P(\mathcal{C})\cap\{x:x_v=1\}$ after projecting out the variable $x_v$. As $\mathcal{C}/ v$ and $\mathcal{C}\setminus v$ are ideal, $P(\mathcal{C}/v)$ and $P(\mathcal{C}\setminus v)$ are integral. Then both $P(\mathcal{C})\cap\{x:x_v=0\}$ and $P(\mathcal{C})\cap\{x:x_v=1\}$ are integral, implying in turn that $x^*$ does not belong to any of these two. So, it follows that $0<x^*_v<1$ for each $v\in V$. Now, consider a nonsingular row submatrix $A$ of $M(\mathcal{C})$ such that $Ax^*={\bf 1}$. Suppose that $V$ has a subset $U$ such that $|C\cap U|=1$ for every member $C$ of $\mathcal{C}$. Let $\chi_U$ denote the characteristic vector of $U$ in $\{0,1\}^{V}$. Since $|C\cap U|=1$ for every member $C$ of $\mathcal{C}$, we have that $M(\mathcal{C})\chi_U={\bf 1}$ and thus $A\chi_U={\bf 1}$. Since $A$ is nonsingular, $Ax={\bf 1}$ has a unique solution, so it follows that $x^*=\chi_U$, a contradiction. Therefore, there is no such subset $U$ of $V$, as required.
\end{proof}

\begin{theorem}\label{idealness-local}
	A multipartite uniform clutter is ideal if and only if all of its localizations are ideal.
\end{theorem}
\begin{proof}
	Let $\mathcal{C}$ be a multipartite uniform clutter whose ground set is partitioned into nonempty parts $V_1,\ldots,V_n$. $(\Rightarrow)$ If $\mathcal{C}$ is ideal, every minor of $\mathcal{C}$ is ideal, and so are all of its localizations. $(\Leftarrow)$ Assume that $\mathcal{C}$ is non-ideal. Then it has a minimally non-ideal minor $\mathcal{C}^\prime:=\mathcal{C}\setminus I/J$ obtained after deleting $I$ and contracting $J$ for some disjoint subsets $I,J\subseteq V_1\cup\cdots\cup V_n$. Observe that $\mathcal{C}\setminus I$ is another multipartite uniform clutter whose ground set is partitioned into nonempty parts $U_1,\ldots, U_n$ where $U_i:=V_i\setminus I$ for $i\in[n]$. In particular, every member $C$ of $\mathcal{C}\setminus I$ satisfies $|C\cap U_i|=1$ for $i\in[n]$. Suppose that $J\cap U_i=\emptyset$ for some $i\in[n]$. Then $|(C-J)\cap U_i|=|C\cap U_i|=1$ for every member $C$ of $\mathcal{C}\setminus I$. As $\mathcal{C}^\prime$ is obtained after contracting $J$ from $\mathcal{C}\setminus I$, we have $|C^\prime\cap U_i|=1$ for every member $C^\prime$ of $\mathcal{C}^\prime$. This contradicts \Cref{mni_no_partitionable_cols} due to our assumption that $\mathcal{C}^\prime$ is minimally non-ideal. Therefore, for each $i\in[n]$, $J\cap U_i\neq\emptyset$, so we have that $J\cap V_i\neq\emptyset$. Then, $\mathcal{C}^\prime$ is a minor of a localization. Therefore, one of $\mathcal{C}$'s localizations is non-ideal, as required.
\end{proof}
In contrast to idealness, even if all localizations have the MFMC property, a multipartite uniform clutter may not have the MFMC property. For example, all localizations of $Q_6=\mult(R_{1,1})$ are isomorphic to the clutter over ground set $\{1,2,3\}$ whose members are $\{1\},\{2\},\{3\}$. The clutter over 3 elements trivially has the MFMC property, but $Q_6$ does not~\cite{Lovasz72a,Seymour77}.

%% file: matroids.tex
\section{Vector spaces, matroids, and sunflower bases}\label{sec:matroid}

Take a prime power $q$, and consider the Galois field $GF(q)$ of order $q$, with additive and multiplicative identities denoted as $0$ and $1$, respectively. 
Take an integer $n\geq 1$, and let $S\subseteq GF(q)^n$ be a vector space over $GF(q)$.

We assume basic knowledge of Matroid Theory throughout this paper (see~\cite{Oxley11} for reference) though we recall matroid theoretic notions as needed. Let us observe a routine way to associate a $GF(q)$-represented matroid to the vector space $S$. Let $A$ be a matrix over $n$ columns with entries in $GF(q)$ such that $S=\{x\in GF(q)^n:Ax={\bf 0}\}$, where the equality in the linear system $Ax={\bf 0}$ holds over $GF(q)$. The \emph{underlying matroid of $S$}, denoted $\mat(S)$, is the matroid represented by $A$ over $GF(q)$. The \emph{dimension} of vector space $S$ is defined as the maximum number of linearly independent vectors in $S$ over $GF(q)$. Note that 
$$\text{the dimension of $S$} = n - \rank(A) = n- \rank\left(\mat(S)\right)$$
where $\rank(A)$ is the matrix rank of $A$ over $GF(q)$ and $\rank\left(\mat(S)\right)$ is the matroid rank of $\mat(S)$ over $GF(q)$.
Although the representation matrix $A$ is not unique for vector space $S$, our terminology suggests that $\mat(S)$ is. The remark below justifies this. 

\begin{RE}\label{RE:supports-matroid}
Take a prime power $q$, and let $S$ be a vector space over $GF(q)$. Then the clutter of circuits of $\mat(S)$ is the set of inclusion-wise minimal members of $\{\supp(x):x\in S,x\neq {\bf 0}\}$.
\end{RE}

Given vectors $v^1,\ldots,v^r\in GF(q)^n$, let $\langle v^1,\ldots,v^r\rangle:=\left\{\sum_{i\in[r]}\lambda_i v^i:~\lambda_i\in GF(q)~~\text{for}~i\in[r]\right\}$, where addition is done over $GF(q)$. The set $\langle v^1,\ldots,v^r\rangle$, which we call the \emph{span} of the vectors, is a vector space over $GF(q)$. 
A \emph{basis} of a vector space $S$ is an inclusion-wise minimal set of vectors whose span is $S$. In this section, we characterize in terms of the underlying matroid when a vector space is spanned by a set of vectors of disjoint supports, or a set of vectors that form a sunflower.

\paragraph{Matroid minors.} We start by arguing that matroid deletions and contractions in $\mat(S)$ correspond to restrictions and projections in $S$. For a matroid $\mathcal{M}$ and disjoint subsets $I,J$ of the ground set of $\mathcal{M}$, we denote by $\mathcal{M}\setminus I/J$ the matroid minor of $\mathcal{M}$ obtained after deleting $I$ and contracting $J$. Let $\mathcal{C}(\mathcal{M})$ denote the clutter of circuits of $\mathcal{M}$. 

\begin{LE}\label{LE:matroid-minor}
Take an integer $n\geq 1$ and a prime power $q$, and let $S\subseteq GF(q)^n$ be a vector space over $GF(q)$. Then $\mat(S)\setminus I/J$ for some disjoint $I,J\subseteq[n]$ is precisely $\mat(S^\prime)$ where $S^\prime\subseteq GF(q)^{n-|I|-|J|}$ is the vector space over $GF(q)$ obtained from $S\cap\{x\in GF(q)^n:\;x_i=0\;\forall i\in I\}$ after dropping coordinates in $I\cup J$.
\end{LE}
\begin{proof}
	It is clear that $S^\prime$ is a vector space over $GF(q)$, so $\mat(S^\prime)$ is well-defined. To show that $\mat(S)\setminus I/J=\mat(S^\prime)$, we will argue that $\mathcal{C}\left(\mat(S)\setminus I/J\right)=\mathcal{C}\left( \mat(S^\prime)\right)$. 
	
	If $\mathcal{C}\left(\mat(S)\setminus I/J\right)=\emptyset$, then every $C\in \mathcal{C}\left(\mat(S)\right)$ intersects $I$, which means that $\supp(x)$ intersects $I$ for every $x\in S-\{\bf 0\}$. This implies that $S^{\prime}=\{\bf 0\}$, in which case $\mathcal{C}\left( \mat(S^\prime)\right)=\emptyset$. Thus we may assume that $\mathcal{C}\left(\mat(S)\setminus I/J\right)\neq \emptyset$. 
	
	Let $C_1\in \mathcal{C}\left(\mat(S)\setminus I/J\right)$. Then there exists $C\in \mathcal{C}\left(\mat(S)\right)$ such that $C\cap I=\emptyset$ and $C_1=C-J$. Then $C=\supp(x)$ for some $x\in S$ by the definition of $\mat(S)$ (see also~\Cref{RE:supports-matroid}). As $C\cap I=\emptyset$, it follows that $x_i=0$ for $i\in I$, which implies that there exists $x^\prime\in S^\prime-\{\mathbf{0}\}$ such that $\supp(x^\prime)=\supp(x)-J=C-J$. So, there exists $C_2\in \mathcal{C}\left( \mat(S^\prime)\right)$ such that $C_2\subseteq C_1$. 
	Therefore, every member of $\mathcal{C}\left(\mat(S)\setminus I/J\right)$ contains a member of $\mathcal{C}\left( \mat(S^\prime)\right)$. 
	
	Let $C_2\in \mathcal{C}\left( \mat(S^\prime)\right)$. Then $C_2=\supp(x^\prime)$ for some $x^\prime\in S^\prime$ by \Cref{RE:supports-matroid}. This implies that there is some $x\in S$ such that $x_i=0$ for $i\in I$ and $\supp(x)-J=\supp(x^\prime)$.
	Since $\supp(x)$ contains a circuit of $\mat(S)$ and $\supp(x)\cap I=\emptyset$, it follows that $C_2=\supp(x^\prime)$ contains a circuit of $\mat(S)\setminus I/J$. Therefore, 	we deduce that $\mathcal{C}\left(\mat(S)\setminus I/J\right)=\mathcal{C}\left( \mat(S^\prime)\right)$, as required.
\end{proof}

\paragraph{Direct sum.} 
Consider matroids $M_1,\ldots,M_\ell$ over pairwise disjoint ground sets $E_1,\ldots,E_\ell$ and independent set families $\mathcal{I}_1,\ldots,\mathcal{I}_\ell$, respectively. The {\it direct sum} of $M_1,\ldots,M_\ell$, denoted $M_1\oplus\cdots\oplus M_\ell$, is the matroid over ground set $E_1\cup\cdots\cup E_\ell$ whose independent set family is $\{I_1\cup\cdots \cup I_\ell: I_i\in \mathcal{I}_i, i\in[\ell]\}$. We shall need the following basic remark about the direct sum of matroids. For the remark, we need to recall two notions. First, a \emph{block} of a graph $G$ is any maximal vertex-induced subgraph of $G$ that is $2$-vertex-connected. Secondly, we denote the \emph{cycle matroid} of a graph $G$ by $\mat(G)$. Finally, we say that a vector space $S$ is the \emph{product} of vector spaces $S_1$ and $S_2$ if $S=\left\{(x,y):~x\in S_1,~y\in S_2\right\}=:S_1\times S_2$.

\begin{LE}\label{direct-sum}
The following statements hold: \begin{enumerate}
\item Let $G$ be a graph, and let $G_1,\ldots,G_k$ be the blocks of $G$. Then $\mat(G)=\mat(G_1)\oplus\cdots\oplus \mat(G_k)$.
\item Take a prime power $q$ and two $GF(q)$-representable matroids $M_1,M_2$ over disjoint ground sets. If $A_1$ and $A_2$ are $GF(q)$-representations of $M_1$ and $M_2$, respectively, then $M_1\oplus M_2$ can be represented by $\left(\begin{smallmatrix}
A_1&0\\
0&A_2
\end{smallmatrix} \right)$.
\item Take a prime power $q$ and a vector space $S$ over $GF(q)$. Then $S=S_1\times S_2$ for some vector spaces $S_1,S_2$ over $GF(q)$ if and only if $\mat(S)=\mat(S_1)\oplus \mat(S_2)$.
\end{enumerate}
\end{LE}
\begin{proof}
{\bf (1), (2)}: See Chapters~4.1 and 4.2 of \cite{Oxley11}. 
{\bf (3)} follows immediately from (2).
\end{proof}

\paragraph{Disjoint supports basis.} 

\begin{LE}\label{disj-supp}
	Take an integer $n\geq 1$ and a prime power $q$, and let $S\subseteq GF(q)^n$ be a vector space over $GF(q)$. Then the following statements are equivalent:
	\begin{enumerate}[(i)]
		\item $\mat(S)=\mat(G)$ where every block of $G$ is either a bridge or a circuit,
		\item $S=\langle v^1,\ldots,v^r\rangle$ where $v^1,\ldots,v^r\in GF(q)^n$ have pairwise disjoint supports, 
		\item 
		$S=S_1\times \cdots \times S_k$ where each $S_i$ has dimension at most $1$.
	\end{enumerate}
\end{LE}
\begin{proof}
It can be readily checked that {\bf (ii)} and {\bf (iii)} are equivalent.
The equivalence of {\bf (i)} and {\bf (iii)} follows from the fact that for a vector space $T$ over $GF(q)$, $T=\{0\}$ if and only if $\mat(T)$ is the cycle matroid of a bridge, and $T$ has dimension $1$ if and only if $\mat(T)$ is the cycle matroid of a circuit.
\end{proof}

\paragraph{Sunflower basis.} 
Given an integer $t\geq 3$, denote by $A_t$ the graph that consists of two vertices and $t$ parallel edges connecting them.

\begin{LE}\label{RE:matroid-block-2}
	Take an integer $n\geq 1$ and a prime power $q$, and let $T\subseteq GF(q)^n$ be a vector space over $GF(q)$. Then $\mat(T)$ is the cycle matroid of a subdivision of $A_t$ for some $t\geq 3$ if and only if $T$ is generated by a sunflower basis.
\end{LE}
\begin{proof}
{\bf $(\Rightarrow)$}: Assume that $\mat(T)=\mat(G)$ where $G$ is a subdivision of $A_t$ for some $t\geq 3$. Notice that $G$ consists of two vertices and $t$ internally vertex-disjoint paths connecting them. Let $P_0,\ldots,P_{t-1}$ denote the paths, and let $E(P_0),\ldots,E(P_{t-1})$ denote their edge sets. Then it follows from \Cref{RE:supports-matroid} that $T$ contains a point whose support is $E(P_0)\cup E(P_i)$. Therefore, $T$ contains $t-1$ points $v^1,\ldots,v^{t-1}$ (in row vectors) of the following form: 
\[
\begin{matrix}
v^1\\
v^2\\
\vdots\\
v^{t-1}
\end{matrix}
\left[
\begin{array}{c|c|c|c|c}
u_1^0 & u^1 & \mathbf{0} & \cdots & \mathbf{0}\\
u_2^0 & \mathbf{0} & u^2 & \cdots & \mathbf{0}\\
\vdots    & \vdots    &   \vdots  & \vdots & \vdots\\
u_{t-1}^0 & \mathbf{0} & \mathbf{0} & \cdots & u^{t-1}
\end{array}
\right]
\]
where $u_1^0,\ldots,u_{t-1}^0\in GF(q)^{|E(P_0)|}$ and $u^i\in GF(q)^{|E(P_i)|}$ for $i\in[n]$ are vectors of nonzero entries. As $T$ is a vector space in $GF(q)^n$, $\mat(T)$ is over $n$ edges, and therefore, $G$ has $n$ edges. Since $G$ is a subdivision of $A_t$, a spanning tree of $G$ has $n-(t-1)$ edges, which means that $\mat(T)=\mat(G)$ has rank $n-(t-1)$. Then the dimension of $T$ is $n-\mat(T) = t-1$, so we have $T=\langle v^1,\ldots,v^{t-1}\rangle$. Now, let us argue that we may assume that $u_1^0=\cdots=u_{t-1}^0$ without loss of generality. As $P_1\cup P_2$ is a circuit of $G$, \Cref{RE:supports-matroid} implies that there is a point $v\in T$ whose support is $E(P_1)\cup E(P_2)$. Then $v$ can be written as $v=\mu_1v^1+\mu_2v^2$ for some $\mu_1,\mu_2\in GF(q)-\{0\}$. As the support of $v$ is $E(P_1)\cup E(P_2)$, we have that $\mu_1 u_1^0+\mu_2 u_2^0=0$, which implies that $u_2^0=\lambda_2 u_1^0$ for some nonzero $\lambda_2$. Similarly, we obtain $u_i^0=\lambda_i u_1^0$ for some nonzero $\lambda_i$ for $i\in[t-1]$, as required. Therefore, after scaling $v^i$'s if necessary, we may assume that $u_1^0=\cdots=u_{t-1}^0$, as required.

{\bf $(\Leftarrow)$}: Suppose $T=\langle v^1,\ldots, v^{t-1}\rangle$ where $v^1,\ldots, v^{t-1}\in GF(q)^n$ are vectors of the following form (in row vectors), after permuting the coordinates, for some $t\geq 3$:
	\[
	\begin{matrix}
	v^1\\
	v^2\\
	\vdots\\
	v^{t-1}
	\end{matrix}
	\left[
	\begin{array}{c|c|c|c|c}
	u^0 & u^1 & \mathbf{0} & \cdots & \mathbf{0}\\
	u^0 & \mathbf{0} & u^2 & \cdots & \mathbf{0}\\
	\vdots    & \vdots    &   \vdots  & \ddots & \vdots\\
	u^0 & \mathbf{0} & \mathbf{0} & \cdots & u^{t-1}
	\end{array}
	\right]
	\]
	for some row vectors $u^0,u^1\ldots, u^{t-1}$ with no zero entries. 
Let $E_i$ be the support of $u^i$ for $i=0,1,\ldots,t-1$. Let $C$ be a circuit of $\mat(T)$. Then $C=\supp(x)$ for some $x\in T$. Let $x=\sum_{i=1}^{t-1} \mu_i v^i$. Then $x$ is of the form
	\[
\begin{matrix}
	x
\end{matrix}
\left[
\begin{array}{c|c|c|c|c}
	\sum_{i=1}^{t-1}\mu_iu^0 & \mu_1 u^1 & \mu_2 u^2 & \cdots & \mu_{t-1} u^{t-1}
\end{array}
\right]
\]
If $C\cap E_0\neq\emptyset$, then it means $\sum_{i=1}^{t-1}\mu_i\neq 0$, and therefore, $C\cap E_0 = E_0$. This implies that the elements in $E_0$ are in series. If $C\cap E_i\neq \emptyset$ for some $1\leq i\leq t-1$, then $\mu_i\neq 0$. This indicates that $C\cap E_i=E_i$, implying in turn that the elements in $E_i$ are in series.

Then consider the case where each $u^i$ is 1-dimensional, under which we have $E_i=\{e_i\}$ is a singleton for $i=0,\ldots, t-1$. Observe that $|\supp(x)|\geq 2$ for any $x\in T$. Then none of $\{e_0\},\{e_1\},\ldots, \{e_{t-1}\}$ is a circuit. However, we know that $\{e_0,e_i\}$ for $i=1,\ldots, t-1$ are circuits of $\mat(T)$ because $v^1,\ldots, v^{t-1}\in T$,  Moreover, $v^i + (q-1)v^j$ for $i\neq j$ has support $\{e_i, e_j\}$, and therefore, $\{e_i,e_j\}$ for distinct $i,j\in \{1,\ldots,t-1\}$ are all circuits. Then $\left\{\{e_i,e_j\}: i,j\in\{0,1,\ldots,t-1\},  i\neq j\right\}$ is the family of circuits of $\mat(T)$ because any subset of the ground set of size at least 3 would contain $\{e_i,e_j\}$ for some $i\neq j$. Therefore, $\mat(T)$ is $\mat(A_t)$. 

In general, as the elements of each $E_i$ are in series, $\mat(T)$ is a series extension of  $\mat(A_t)$, which is equivalent to the cycle matroid of a subdivision of $A_t$, as required.
\end{proof}

\paragraph{Putting it altogether.}

\begin{CO}\label{almost-disj-}
	Take an integer $n\geq 1$ and a prime power $q$, and let $S\subseteq GF(q)^n$ be a vector space over $GF(q)$. Then the following statements are equivalent:
	\begin{enumerate}[(i)]
		\item $\mat(S)=\mat(G)$ where every block of $G$ is a bridge, a circuit, or a subdivision of $A_t$ for some $t\geq 3$,
		\item $S=S_1\times\cdots\times S_k$ where each $S_i$ has dimension at most $1$, or admits a sunflower basis.\end{enumerate}
\end{CO}

%% file: mfmc-odd.tex
\section{The MFMC property and odd prime powers}\label{sec:mfmc-odd}

Let $q$ be a power of a prime number $p$. Recall that we denote by $0$ and $1$ the additive and multiplicative identities of $GF(q)$. Then there must exist an integer $\ell$ such that $a+a+\cdots+a$ ($\ell$ times) equals $0$ for all $a\in GF(q)$, and in fact, the smallest of such integers is $p$. Here, $p$ is often referred to as the {\it characteristic} of $GF(q)$. Throughout this paper, we denote by $-v$ and $v^{-1}$ the additive and multiplicative inverses of $v$ for each $v\in GF(q)-\{0\}$. 

In this section, we prove three lemmas that are useful for this section as well as the next three. Then we prove \Cref{mfmc} that characterizes when the multipartite uniform clutter of a vector space has the MFMC property. Lastly, we prove \Cref{q odd} for the case when $q$ is an odd prime power.

\begin{LE}\label{q-space:triple}
	Take an integer $n\geq 3$ and $n$ non-empty sets  $V_1,\ldots, V_n$, and let $S\subseteq V_1\times\cdots\times V_n$. If $\mult(S)$ contains no $\Delta_3$ as a minor, then for any distinct $a,b,c\in S$ and distinct $i,j,k\in[n]$ such that
	\begin{equation}\label{triple-condition}
	a_i=b_i\neq c_i,\quad b_j=c_j\neq a_j,\quad c_k=a_k\neq b_k\tag{$\star$},
	\end{equation}
	there exists $d\in S-\{a,b,c\}$ that satisfies the following:
	\begin{enumerate}[(1)]
		\item $d_\ell\in\{a_\ell,b_\ell,c_\ell\}$ for all $\ell\in[n]$, and
		\item at least two of $d_i=c_i$, $d_j=a_j$, and $d_k=b_k$ hold. 
	\end{enumerate}
\end{LE}
\begin{proof}
	Let $V$ denote the ground set of $\mult(S)$. We may assume that there exist three distinct points $a,b,c\in S$ satisfying~\eqref{triple-condition} for some distinct $i,j,k\in[n]$. Take subsets $I,J$ of $[n]$ as follows:
	\[
	I=V-\{a_\ell,b_\ell,c_\ell:\ell\in[n]\}\quad\text{and}\quad J=\{a_\ell,b_\ell,c_\ell:\ell\in[n]-\{i,j,k\}\}.
	\] 
	We will show that if $d\in S-\{a,b,c\}$ satisfying (1) and (2) does not exsit, then $\mult(S)\setminus I/J$ contains $\Delta_3$ as a minor. 
	
	Notice that $\mult(S)\setminus I$ is $\mult(R_0)$ where $R_0=\left\{v\in S:~v_\ell\in\{a_\ell,b_\ell,c_\ell\}~\text{for}~\ell\in[n]\right\}$ and that each member of $\mult (R_0)$ is $\{v_1,\ldots,v_n\}$ for some $v\in S$. Furthermore, each $v\in R_0$ satisfies $\{v_1,\ldots,v_n\} - J=\{v_i,v_j,v_k\}$, so $\{v_1,\ldots,v_n\} - J$ remains minimal after contracting $J$ from $\mult(R_0)$. This in turn implies that $\mult(R_0)/J$ is equal to $\mult(R)$ where 
	$
	R:=\left\{(v_i,v_j,v_k):\;v\in S,\;v_\ell\in\{a_\ell,b_\ell,c_\ell\}~\text{for}~\ell\in[n]\right\}.
	$
	So, $\mult(S)\setminus I/J=\mult(R)$. By definition, $R$ contains points $(a_i,a_j,a_k)$, $(b_i,b_j,b_k)$, and $(c_i,c_j,c_k)$ that are obtained from $a,b,c$. Suppose that there is no $d\in S-\{a,b,c\}$ that satisfies (1) and (2). Let $d\in S$ with $d_\ell\in\{a_\ell,b_\ell,c_\ell\}$ for $\ell\in [n]$. Since $d$ satisfies (1), $d$ does not satisfy (2). Then $(d_i,d_j,d_k)$ can be $(c_i,b_j,c_k)$, $(a_i,a_j,c_k)$, $(a_i,b_j,b_k)$, or $(a_i,b_j,c_k)$, implying in turn that
	\[
	R\subseteq\left\{(a_i,a_j,a_k),(b_i,b_j,b_k),(c_i,c_j,c_k),(c_i,b_j,c_k), (a_i,a_j,c_k), (a_i,b_j,b_k),(a_i,b_j,c_k)
	\right\}.
	\]
	To argue that $\mult(R)$ contains $\Delta_3$ as a minor, let us look at the incidence matrix of $\mult(R)$:
	\[
	\bordermatrix{& a_i & \overbrace{\mathbf{c_i}} & \overbrace{\mathbf{a_j}} & b_j & c_k & \overbrace{\mathbf{b_k}}\cr
		a& 1 & \mathbf{0} & \mathbf{1} & 0 & 1 & \mathbf{0}\cr
		b& 1 & \mathbf{0} & \mathbf{0} & 1 & 0 & \mathbf{1}\cr
		c& 0 & \mathbf{1} & \mathbf{0} & 1 & 1 & \mathbf{0}\cr
		&   &   &   &\vdots&&  \cr	
	}
	\]
	Observe that a row of $M(\mult(R))$ other than the ones for $a,b,c$, if any, has at least two ones in the columns for $a_i,b_j,c_k$. So, after contracting the columns for $c_i,a_j,b_k$ and removing non-minimal rows, the resulting incidence matrix is precisely $M(\Delta_3)$.
	This implies that we obtain $\Delta_3$ after contracting $c_i,a_j,b_k$ from $\mult(R)$, a contradiction to the assumption that $\mult(S)$ has no $\Delta_3$ minor.
\end{proof}

\begin{LE}\label{LE:non-disjoint-supports}
	Take an integer $n\geq 1$ and a prime power $q$, and  let $S\subseteq GF(q)^n$ be a vector space over $GF(q)$. If $S$ does not admit a basis with vectors of pairwise disjoint supports, then $\mult(S)$ contains $\Delta_3$ or $Q_6$ as a minor. Moreover, if $q$ is an odd prime power, then $\mult(S)$ contains $\Delta_3$ as a minor.
\end{LE}
\begin{proof}
	Assume that $S$ does not admit a basis with vectors of pairwise disjoint supports. We will show that if $\mult(S)$ does not contain $\Delta_3$ as a minor, then $q$ is a power of $2$ and $\mult(S)$ contains $Q_6$ as a minor.
	
	Assume that $S$ contains no $\Delta_3$ as a minor. Let $v^1,\ldots,v^r\in GF(q)^n$ be a basis of $S$. After elementary arithmetic operations over $GF(q)$, we may assume that for each $i=1,\ldots,r$,
	\[
	v^i_i=1\quad\text{and}\quad v^i_j=0\quad\forall j\in[r]-\{i\}
	\] 
	Since there is no basis of $S$ with vectors of pairwise disjoint supports, we may assume that $v^1_{r+1},v^2_{r+1}\neq0$. This in turn implies that $n\geq 3$. Let $x$ and $y$ be the multiplicative inverses of $v^1_{r+1}$ and $v^2_{r+1}$ in $GF(q)$, respectively. Let $a:=\mathbf{0}\in GF(q)^n$, $b:=xv^1$, and $c:=yv^2$. Notice that $a,b,c\in S$ and that $a,b,c$ satisfy
	\[(a_1,a_2,a_{r+1})=(0,0,0),\quad (b_1,b_2,b_{r+1})=(x,0,1),\quad (c_1,c_2,c_{r+1})=(0,y,1).\]
	Now we consider $R=\{d\in S:d_j\in\{a_j,b_j,c_j\}~\text{for}~j\in[n]\}$.
	\begin{claim}
		$R\subseteq \left\{\lambda_1v^1+\lambda_2v^2:\;\lambda_1\in\{0,x\},\;\lambda_2\in\{0,y\}\right\}$.
	\end{claim}
	\begin{cproof}
		Let $u\in R$. Then $u=\sum_{j=1}^r\lambda_j v^j$ for some $\lambda_1,\ldots,\lambda_r\in GF(q)$. Since $a_j=b_j=c_j=0$ for $j=3,\ldots,r$, it follows that $u_3=\cdots=u_r=0$, which implies that $\lambda_3=\cdots=\lambda_r=0$ and so $u=\lambda_1v^1+\lambda_2v^2$. Notice that $\lambda_1\in\{0,x\}$ and $\lambda_2\in\{0,y\}$, because $a_1,b_1,c_1\in\{0,x\}$ and $a_2,b_2,c_2\in\{0,y\}$.
	\end{cproof}
	
	\begin{claim}
		$q$ is a power of $2$ and $R=\left\{\lambda_1v^1+\lambda_2v^2:\;\lambda_1\in\{0,x\},\;\lambda_2\in\{0,y\}\right\}$.
	\end{claim}
	\begin{cproof}
		By \Cref{q-space:triple}, $R$ contains a $d\not\in\{a,b,c\}$ such that
		$
		(d_1,d_2,d_{r+1})=(0,y,0),\;(x,0,0),\;(x,y,1)$, or $(x,y,0).
		$
		By Claim~1, $d\in\left\{\lambda_1v^1+\lambda_2v^2:\;\lambda_1\in\{0,x\},\;\lambda_2\in\{0,y\}\right\}$. As $d\neq a,b,c$, it must be that $xv^1+yv^2=d$, so $xv^1+yv^2\in R$. In particular, $R=\left\{\lambda_1v^1+\lambda_2v^2:\;\lambda_1\in\{0,x\},\;\lambda_2\in\{0,y\}\right\}$. Since $d=xv^1+yv^2$, we obtain $(xv^1+yv^2)_{r+1}=1+1=d_{r+1}\in\{0,1\}$. Since $1\neq 0$, we have $1+1=0$, so $q$ is a power of~2, as required. 
	\end{cproof}
	
	By Claim~2, $R=\left\{\lambda_1v^1+\lambda_2v^2:\;\lambda_1\in\{0,x\},\;\lambda_2\in\{0,y\}\right\}$, so the projection of $R$ onto the space of coordinates $1,2,r+1$ is precisely $R_{1,1}$. Since $\mult(R_{1,1})=Q_6$, $\mult(S)$ has $Q_6$ as a minor by \Cref{projection}. So, we have shown that if $\mult(S)$ has no $\Delta_3$ as a minor, then $q$ is a power of 2 and $\mult(S)$ contains $Q_6$ as a minor, as required.
\end{proof}

\begin{LE}\label{disj->ideal}
Take an integer $n\geq 1$ and a prime power $q$, and let $S\subseteq GF(q)^n$ be a vector space over $GF(q)$. If $S$ has a basis with vectors of pairwise disjoint supports, then $\mult(S)$ has the MFMC property, and is therefore ideal.
\end{LE}
\begin{proof}
Assume that $S$ has a basis of vectors with pairwise disjoint supports. Then we may assume that $S=\langle u^1\rangle\times\cdots\times\langle u^r\rangle\times \{\mathbf{0}\}$ for some vectors $u^1,\ldots,u^r$ with no zero entries over $GF(q)$, by \Cref{disj-supp}. Subsequently, $\mult(S)=\mult(\langle u^1\rangle)\times\cdots\times\mult(\langle u^r\rangle)\times\mult(\{\mathbf{0}\})$, and to prove $\mult(S)$ has the MFMC property, it suffices to argue that $\mult(\langle u^i\rangle)$ for $i\in[r]$ and $\mult(\{\mathbf{0}\})$ have the MFMC property, by \Cref{set-product}. First, notice that $\mult(\{\mathbf{0}\})$ has only one member, so it clearly has the MFMC property. In fact, we can argue that each $\mult(\langle u^i\rangle)$ has pairwise disjoint members as well. Notice that for any distinct $x,y\in GF(q)$, $xu^i$ and $yu^i$ do not have common coordinates, implying in turn that the members of $\mult\left(\langle u^i\rangle\right)$ corresponding to $xu^i$ and $yu^i$ are disjoint. That means that the members of $\mult\left(\langle u^i\rangle\right)$ are pairwise disjoint, implying in turn that it has the MFMC property, thereby proving that $\mult(S)$ has the MFMC property, as required.
\end{proof}

Having proved Lemmas~\ref{LE:non-disjoint-supports} and~\ref{disj->ideal}, we are now ready to show~\Cref{mfmc}. The basic flow of our proof is as follows. \Cref{disj->ideal} shows that if a vector space $S$ has a basis with vectors of pairwise disjoint supports, then $\mult(S)$ has the MFMC property. Conversely, \Cref{LE:non-disjoint-supports} argues that if a vector space $S$ does not admit such a basis, then $\mult(S)$ has some minors certifying that the clutter does not have the MFMC property. More details are explained in the proof as follows.

\begin{proof}[{\bf Proof of \Cref{mfmc}}]
{\bf (iii)}$\Rightarrow${\bf (ii)} follows from \Cref{LE:non-disjoint-supports}. 
{\bf (ii)}$\Rightarrow${\bf (i)} follows from \Cref{disj->ideal}.
{\bf (i)}$\Rightarrow${\bf (iii)}: Assume that $\mult(S)$ has the MFMC property. $\Delta_3$ is a non-ideal clutter, so it does not have the max-flow min-cut property. Recall that $Q_6$ is the clutter of triangles in $K_4$. Notice that the minimum number of edges required to intersect every triangle in $K_4$ is two and that the maximum number of disjoint triangles in $K_4$ is one. This implies that $\tau(Q_6,\mathbf{1})=2$ and $\nu(Q_6,\mathbf{1})=1$, so $Q_6$ does not have the max-flow min-cut property. Like idealness, the MFMC property is a minor-closed property~\cite{Seymour77}. Therefore, a clutter with the MFMC property contains none of $\Delta_3,Q_6$ as a minor, implying in turn that $\mult(S)$ has none of $\Delta_3, Q_6$ as a minor.
\end{proof}

The proof of~\Cref{q odd} works similarly as that of \Cref{mfmc}. The additional component is that when $q$ is an odd prime power and a vector space $S$ over $GF(q)$ does not admit a basis with vectors of pairwise disjoint supports, then $\mult(S)$ has a non-ideal minor due to \Cref{LE:non-disjoint-supports}.

\begin{proof}[{\bf Proof of \Cref{q odd}}]
Take an integer $n\geq 1$ and an odd prime power $q$, and let $S\subseteq GF(q)^n$ be a vector space over $GF(q)$. Since $\Delta_3$ is non-ideal, direction {\bf (i)}$\Rightarrow${\bf (iii)} is clear. Direction {\bf (iii)}$\Rightarrow${\bf (ii)} follows from \Cref{LE:non-disjoint-supports}, and \Cref{disj->ideal} shows direction {\bf (ii)}$\Rightarrow${\bf (i)}. Therefore, {\bf (i)}--{\bf(iii)} are equivalent. 
\end{proof}

%% file: p=2-structure.tex
\section{Fields of characteristic $2$: a structure theorem}\label{sec:p=2-structure}

In this section, we prove \Cref{LE:no U24 K4/e} which provides an important tool for characterizing the idealness of $\mult(S)$ where $S$ is a vector space over $GF(2^k)$ for $k\geq2$. To be more precise, \Cref{LE:no U24 K4/e} characterizes the structure of the underlying matroid $\mat(S)$ when $\mult(S)$ is ideal and thus has no $\Delta_3$ as a minor.

\begin{LE}\label{PR:U24}
	Let $q$ be a power of 2, and let $S\subseteq GF(q)^4$ be a vector space over $GF(q)$. If $\mat(S)$ is isomorphic to $U_{2,4}$, then $\mult(S)$ has $\Delta_3$ as a minor.
\end{LE}
\begin{proof}
	Suppose for a contradiction that $\mult(S)$ has no $\Delta_3$ as a minor. Since the rank of $U_{2,4}$ is 2, the dimension of $S$ is $4-2=2$. Let $v^1,v^2\in GF(q)^4$ be two generators of $S$. By elementary row operations, we may assume that $(v^1_1,v^1_2)=(1,0)$ and $(v^2_1,v^2_2)=(0,1)$. Then
	\[
	\begin{matrix}
	v^1\\
	v^2
	\end{matrix}
	\left[
	\begin{array}{cc|cc}
	1 & 0 & x & y\\
	0 & 1 & z & w
	\end{array}
	\right]
	\]
	where $x,y,z,w\in GF(q)$. Each circuit of $U_{2,4}$ has size 3, so $x,y,z,w\neq 0$. Then $a:=(-x^{-1}z)v^1$, $b:=v^2$, $c:=a+b$ are vectors in $S$. Observe that
	\[
	\begin{matrix}
	a\\
	b\\
	c
	\end{matrix}
	\left[
	\begin{array}{c|c|c|c}
	-x^{-1}z & 0 & -z & -x^{-1}yz\\
	0 & 1 & z & w\\
	-x^{-1}z & 1 & 0 & -x^{-1}yz +w
	\end{array}
	\right]
	\]
	and that $a_1=c_1\neq b_1$, $b_2=c_2\neq a_2$. We also have that $a_3=b_3\neq c_3$, because $q$ being a power of 2 implies $z+z=0$ and $z=-z$. By \Cref{q-space:triple}, there is a vector $d\in GF(q)^4$ that satisfies at least two of $d_1=b_1=0$, $d_2=a_2=0$, $d_3=c_3=0$ and satisfies $d_4\in\{-x^{-1}yz,w,-x^{-1}yz+w\}$. But then the support of $d$ has size at most 2. Since every circuit of $U_{2,4}$ has size 3, $d=\mathbf{0}$, and therefore, $d_4=-x^{-1}yz+w=0$. This implies the support of $c$ has size 2, a contradiction.
\end{proof}

\begin{wrapfigure}{r}{0.25\textwidth}
	\begin{center}
		\begin{tikzpicture}
		\draw[fill=black] (0,0) circle (3pt);
		\draw[fill=black] (4,0) circle (3pt);
		\draw[fill=black] (2,2) circle (3pt);
		\node[-] (1) at (0,0) {};
		\node[-] (2) at (2,2) {};
		\node[-] (3) at (4,0) {};
		
		\node[-,label=below:1] (a) at (2,0) {};
		\node[-,label=above:2] (b) at (1.3,0.7) {};
		\node[-,label=above:3] (c) at (2.7,0.7) {};
		\node[-,label=above:4] (d) at (0.7,1.3) {};
		\node[-,label=above:5] (e) at (3.3,1.3) {};
		\path[thick, bend left] (1) edge (2);
		\path[thick, bend right] (1) edge (2);
		\path[thick, bend left] (3) edge (2);
		\path[thick, bend right] (3) edge (2);
		\path[thick] (1) edge (3);
		\end{tikzpicture}
		\caption{$K_4/e$}\label{fig:K4/e}
	\end{center}
\end{wrapfigure}

\paragraph{Graph minors.}
We say that a graph $H$ is a {\it graph minor} of a graph $G$ if $H$ can be obtained from $G$ after a series of edge deletions, edge contractions, and deletions of isolated vertices. If $G$ is connected, then $H$ is a graph minor of $G$ if and only if for some disjoint subsets $E_1,E_2$ of $E(G)$, we can obtain $H$ from $G$ by deleting $E_1$ and contracting $E_2$. 
It is well-known that if $H$ is a graph minor of $G$, then $\mat(H)$ is a matroid minor of $\mat(G)$~(see Chapter 3.2 in~\cite{Oxley11}).

$K_4$ is the complete graph on $4$ vertices, and we denote by $K_{4}/e$ what is obtained from $K_4$ after contracting an edge from it (see Figure~\ref{fig:K4/e}).

\begin{LE}\label{PR:K4/e}
	Let $q=2^k$ for some $k\geq 2$, and let $S\subseteq GF(q)^5$ be a vector space over $GF(q)$. If $\mat(S)$ is isomorphic to $\mat(K_4/e)$, then $\mult(S)$ has $\Delta_3$ as a minor.
\end{LE}
\begin{proof}
	In Figure~\ref{fig:K4/e}, we can see that the fundamental circuits of $K_4/e$ with respect to spanning tree $\{4,5\}$ are $\{1,4,5\}$, $\{2,4\}$, $\{3,5\}$. Pick vectors $v^1,v^2,v^3\in S$ whose supports are the three circuits. Notice that these vectors are linearly independent. Since the dimension of $S$ is $5-2=3$, vectors $v^1,v^2,v^3$ generate $S$. After elementary row operations, $S$ is generated by the 3 vectors $v^1$, $v^2$, $v^3$ of the following forms:
	\[
	\begin{matrix}
	v^1\\
	v^2\\
	v^3
	\end{matrix}
	\left[
	\begin{array}{ccc|cc}
	1 & 0 & 0 & x & y\\
	0 & 1 & 0 & z & 0\\
	0 & 0 & t & 0 & w
	\end{array}
	\right]
	\]
	where $t,x,y,z,w\neq 0$. Since $q>2$, we may assume that $z$ and $w$ are distinct nonzero elements in $GF(q)$. Now consider the restriction $S^\prime$ of $S$ defined as follows:
	\[
	S^\prime:=S\cap\left\{x\in GF(q)^5:\;x_1\in\{0,z,w\},\;x_2\in\{0,x\},\;x_3\in\{0,ty\}\right\}.
	\]
	We will show that $\mult(S^\prime)$ has $\Delta_3$ as a minor. Then as $S^\prime$ is a restriction of $S$, it follows from \Cref{projection} that $\mult(S)$ also has $\Delta_3$ as a minor. Notice that 
	\[
	S^\prime=\left\{\sum_{i=1}^3\lambda_iv^i:\; \lambda_1\in\{0,z,w\},\;\lambda_2\in\{0,x\},\;\lambda_3\in\{0,y\}\right\}.
	\]
	Consider three distinct points $a:=zv^1$, $b:=wv^1$, $c:=xv^2+yv^3$ in $S^\prime$:
	\[
	\begin{matrix}
	a\\
	b\\
	c
	\end{matrix}
	\left[
	\begin{array}{ccc|cc}
	z & 0 & 0 & zx & zy\\
	w & 0 & 0 & wx & wy\\
	0 & x & ty & zx & wy
	\end{array}
	\right]
	\]
	As $z\neq w$, we have that $c_4=a_4\neq b_4$ and $b_5=c_5\neq a_5$. We also have $a_3=b_3\neq c_3$, because $ty\neq 0$. Suppose for a contradiction that $\mult(S^\prime)$ has no $\Delta_3$ as a minor. By \Cref{q-space:triple}, there is $d\in S^\prime-\{a,b,c\}$ that satisfies
	\begin{enumerate}[(1)]
		\item $d_1\in\{0,z,w\}$, $d_2\in\{0,x\}$, $d_3\in\{0,ty\}$, $d_4\in\{zx,wx\}$, $d_5\in\{zy,wy\}$, and
		\item at least two of $d_3=ty$, $d_4=wx$, $d_5=zy$ hold.
	\end{enumerate}
	The points of $S^\prime-\{a,b,c\}$ are the following:
	\[
	S^\prime-\{a,b,c\}=\left\{
	\begin{array}{c|c|c}
	(0,0,0,0,0)  & (0, x, 0, zx, 0) & (0,0,ty,0,wy)\\
	(z,x,0,0,zy) & (z,0,ty,zx,(z+w)y) & (w,x,0,(z+w)x,wy)\\
	(w,0,ty,wx,0) & (z, x, ty, 0, (z+w)y) & (w, x, ty, (z+w)x, 0)\\
	\end{array}
	\right\}.
	\]
	Since $z,w\neq 0$ and $z\neq w$, $(z+w)x\not\in \{zx,wx\}$ and $(z+w)y\not\in \{zy,wy\}$. Since $z,w,x,y\neq 0$, $0\not\in \{zx,wx\}$ and $0\not\in \{zy,wy\}$. This indicates that no point in $S^\prime-\{a,b,c\}$ satisfies condition (1), a contradiction. Therefore, $\mult(S^\prime)$ has $\Delta_3$ as a minor, and so does $\mult(S)$, as required.
\end{proof}

How does a graph with no $K_4/e$ graph minor look like? We have the following result. 

\begin{LE}\label{PR:graph-K4/e}
	Let $G=(V,E)$ be a connected graph. If $G$ contains no $K_4/e$ as a graph minor, then each block of $G$ is a bridge, a circuit, or a subdivision of $A_t$ for some $t\geq 3$.
\end{LE}
\begin{proof}
See \S\ref{sec:graph-K4/e} in the appendix.
\end{proof}

We call a graph a {\it series-parallel network} if each of its blocks is a series-parallel graph.

\begin{theorem}[\cite{Brylawski71}]\label{Brylawski}
	Let $M$ be a matroid. Then the following statements are equivalent:
	\begin{enumerate}[(i)]
		\item $M$ contains none of $U_{2,4}$ and $\mat(K_4)$ as a matroid minor,
		\item $M$ is the cycle matroid of a series-parallel network.
	\end{enumerate}
\end{theorem}

\begin{theorem}\label{LE:no U24 K4/e}
	Let $q=2^k$ for some $k\geq 2$, and let $S$ be a vector space over $GF(q)$. If $\mult(S)$ has no $\Delta_3$ as a minor, then for some $k\geq 1$, $\mat(S)=M_1\oplus \cdots\oplus M_k$, where each $M_i$ is the cycle matroid of a bridge, a circuit, or a subdivision of $A_t$ for some $t\geq 3$.
\end{theorem}
\begin{proof}
	Assume that $\mult(S)$ has no $\Delta_3$ as a minor.  Suppose for a contradiction that $\mat(S)$ contains  $U_{2,4}$ or $\mat(K_4/e)$ as a matroid minor. This in turn implies that there exists $S^\prime$ obtained from $S$ after a series of restrictions and projections such that $\mat(S^\prime)$ is isomorphic to $U_{2,4}$ or $\mat(K_4/e)$ by \Cref{LE:matroid-minor}. Here, $\mult(S^\prime)$ contains $\Delta_3$ as a minor by Lemmas \ref{PR:U24} and~\ref{PR:K4/e}.  As $\mult(S^\prime)$ is a minor of $\mult(S)$ due to \Cref{projection}, it follows that $\mult(S)$ also contains a $\Delta_3$ as a minor, a contradiction. Hence, $\mat(S)$ contains  none of $U_{2,4}$ and $\mat(K_4/e)$ as a matroid minor.
	As $\mat(K_4/e)$ is a matroid minor of $\mat(K_4)$, \Cref{Brylawski} implies that $\mat(S)$ is the cycle matroid of a series-parallel network not containing $K_4/e$ as a graph minor. Then by \Cref{PR:graph-K4/e}, each block of the graph is a subdivision of $A_t$ for some $t\geq 3$, a bridge, or a circuit. 
	So, the assertion follows from \Cref{direct-sum}, as required.
\end{proof}

Suppose $S$ is a vector space over $GF(2^k)$ for some $k\geq 2$.
By \Cref{LE:no U24 K4/e}, if $\mult(S)$ has no $\Delta_3$ as a minor, then the underlying matroid can be decomposed as the direct sum of some structured graphic matroids. Then it follows from \Cref{almost-disj-} that $S$ can be represented as $S=S_1\times\cdots\times S_k$ where each $S_i$ has dimension at most~1, or admits a sunflower basis. Then the idealness of $S$ is determined by $S_1,\ldots, S_k$ according to \Cref{set-product}. In particular, we need to understand the case where $S_i$ is a vector space that admits a sunflower basis. In the next section, we provide tools for characterizing when the multipartite uniform clutter of a vector space that admits a sunflower basis is ideal.

%% file: p=2-At.tex
\section{Fields of characteristic $2$: a study of the localizations for $A_t$}\label{sec:p=2-At}

Suppose $S$ is a vector space over $GF(2^k)$ for $k\geq 2$.
At the end of \Cref{sec:p=2-structure}, we discussed that understanding vector spaces that admit a sunflower basis is the key to characterizing when $\mult(S)$ is ideal. In this section, we consider the case when $\mat(S)=\mat(A_n)$ for some $n\geq 3$, where $A_n$ denotes the graph that consists of two vertices and $n$ parallel edges connecting them. Recall that by \Cref{idealness-local}, $\mult(S)$ is ideal if and only if all its localizations are ideal. In this section, we prove three lemmas on properties of localizations of $\mult(S)$.

\begin{LE}\label{RE:parallel-generators}
	Take an integer $n\geq 3$ and a prime power $q$, and let $S\subseteq GF(q)^n$ be a vector space over $GF(q)$. Then $\mat(S)=\mat(A_n)$ if and only if $S\cong \{x\in GF(q)^n:\;x_1+\cdots+x_n=0\}$.
\end{LE}

\begin{proof} 
	Let $\{1,2,3\ldots,n\}$ denote the edge set of $A_n$. Then  $\{1,2\},\{1,3\},\ldots,$ $\{1,n\}$ are circuits of $\mat(A_n)$. {\bf ($\Leftarrow$)}: Let $\mathcal{S}$ be the clutter of the minimal supports of the points in $S-\left\{\mathbf{0}\right\}$. Then $\mathcal{S}=\left\{\{i,j\}:i\neq j\right\}$, so $\mat(S)=\mat(A_n)$ by \Cref{RE:supports-matroid}. {\bf ($\Rightarrow$)}: Since $\mat(S)=\mat(A_n)$, $S$ contains $n-1$ points $u^1,\ldots,u^{n-1}$ whose supports are $\{1,2\}$, $\{1,3\},\ldots,\{1,n\}$, respectively. Notice that $u^1,\ldots, u^{n-1}$ are linearly independent over $GF(q)$, so the dimension of $S$ is at least $n-1$. On the other hand, the dimension is less than $n$, because $S\neq GF(q)^n$. Thus, $S=\langle u^1,\ldots, u^{n-1}\rangle$. After scaling the $u^i$s, if necessary, we may assume that the first coordinate of each $u^i$ is 1. Hence, $u^1,\ldots, u^{n-1}$ are of the form displayed below (left), where $\lambda_1,\ldots,\lambda_{n-1}\in GF(q)-\{0\}$. Notice that $\{x\in GF(q)^n:\;x_1+\cdots+x_n=0\}=\langle v^1,\ldots,v^{n-1}\rangle$ where $v^1,\ldots,v^{n-1}$ are displayed below (right):
	\[
	\begin{matrix}
	u^1\\
	u^2\\
	\vdots\\
	u^{n-1}
	\end{matrix}
	\left[
	\begin{array}{ccccc}
	1 & \lambda_1 & 0 & \cdots & 0\\
	1 & 0 & \lambda_2 & \cdots & 0\\
	\vdots    & \vdots    &   \vdots  & \vdots & \vdots\\
	1 & 0 & 0 & \cdots & \lambda_{n-1}
	\end{array}
	\right]
	\qquad \begin{matrix}
	v^1\\
	v^2\\
	\vdots\\
	v^{n-1}
	\end{matrix}
	\left[
	\begin{array}{ccccc}
	1 &-1 & 0 & \cdots & 0\\
	1 & 0 & -1 & \cdots & 0\\
	\vdots    & \vdots    &   \vdots  & \vdots & \vdots\\
	1 & 0 & 0 & \cdots & -1
	\end{array}
	\right],
	\] implying in turn that
	$
	\{x\in GF(q)^n:\;x_1+\cdots+x_n=0\}=\left\{(x_1,-\lambda_1^{-1}x_2,-\lambda_2^{-1}x_3,\ldots,-\lambda_{n-1}^{-1}x_n):\;x\in S\right\}.
	$
	Therefore, $S\cong \{x\in GF(q)^n:\;x_1+\cdots+x_n=0\}$, as required.
\end{proof}

By \Cref{RE:parallel-generators}, we may focus on the set $$S=\{x\in GF(q)^n:\;x_1+\cdots+x_n=0\}$$
to understand vector spaces whose underlying matroids are $\mat(A_n)$.
Recall that a localization of $\mult(S)$ with respect to $\alpha\in GF(q)^n$, denoted $\local(S,\alpha)$, is the minor of $\mult(S)$ after contracting the elements corresponding to $\alpha$ (see \Cref{sec:cuboid}). $\mult(S)$ is defined over ground set $V_1\cup\cdots\cup V_n$ where each $V_i$ is a copy of $GF(q)$, and $\local(S,\alpha)$'s ground set is given by $U_1\cup\cdots U_n$ where $U_i = V_i -\{\alpha_i\}$. 
The following lemma provides a characterization of the members of $\local(S,\alpha)$ for any $\alpha\not\in S$.

\begin{LE}\label{LE:hyperedges}
	Take an integer $n\geq 3$. Let $q$ be a power of 2, and let $\alpha\in GF(q)^n$ with $\sigma:=\alpha_1+\cdots+\alpha_n\neq 0$. Let $C\subseteq U_1\cup\cdots \cup U_n$ where $U_i=GF(q)-\{\alpha_i\}$. Then the following statements are equivalent:
	\begin{enumerate}[(i)]
		\item $C$ is a member of $\local(S,\alpha)$.
		\item  $C$ contains at most one element in $U_i$ for each $i\in[n]$ and $\sum(v:\;v\in C) = \sigma+\sum(\alpha_i:\;C\cap U_i\neq \emptyset)$.
	\end{enumerate} 
\end{LE}
\begin{proof}
	{\bf (i)$\Rightarrow$(ii)} There exists $x=(x_1,\ldots,x_n)\in S$ such that $C=\{x_1,\ldots,x_n\}-\{\alpha_1,\ldots,\alpha_n\}$. Then $C\cap U_i=\{x_i\}-\{\alpha_i\}$, implying that $C\cap U_i$ has at most one element. Without loss of generality, we may assume that $x=(x_1,\ldots,x_k,\alpha_{k+1},\ldots,\alpha_n)$ and $x_1\neq \alpha_1,\ldots,x_k\neq \alpha_k$ for some $1\leq k\leq n$. Then $C=\{x_1,\ldots,x_k\}$. Since $x\in S$, we have
	\[
	\sum_{i=1}^nx_i=\sum_{i=1}^kx_i+\sum_{j=k+1}^n\alpha_j=0.
	\]
	As the characteristic of $GF(q)$ is 2, $\sum_{i=1}^kx_i=-\sum_{i=1}^kx_i$, implying in turn that $\sum_{i=1}^kx_i=\sum_{j=k+1}^n\alpha_j$. As $\sum_{i=1}^n\alpha_i=\sigma$, we also get $\sum_{j=k+1}^n\alpha_j=\sigma+\sum_{i=1}^k\alpha_i$, and therefore, we obtain $\sum_{i=1}^kx_i=\sigma+\sum_{i=1}^k\alpha_i$, as required.
	
	{\bf (i)$\Leftarrow$(ii)} Without loss of generality, we may assume that $C=\{x_1,\ldots,x_k\}$ where $x_i\in U_i$ for $i\in[k]$. Then $\left\{x_1,\ldots,x_k\right\}\cap \left\{\alpha_1,\ldots,\alpha_n\right\}=\emptyset$. Since $\sum_{i=1}^kx_i=\sigma+\sum_{i=1}^k\alpha_i$, we have $\sum_{i=1}^kx_i+\sum_{j=k+1}^n\alpha_j=\sigma+\sum_{i=1}^n\alpha_i=0$, implying in turn that $(x_1,\ldots,x_k,\alpha_{k+1},\ldots,\alpha_n)\in S$. As $C=\{x_1,\ldots,x_k,\alpha_{k+1},\ldots,\alpha_n\}-\{\alpha_1,\ldots,\alpha_n\}$, it follows that $C$ is a member of $\local(S,\alpha)$, as required. 
\end{proof}

Using \Cref{LE:hyperedges}, we can show the following lemma providing a characterization of the members of size 1 and 2 in $\local(S,\alpha)$ for $\alpha\not\in S$.

\begin{LE}\label{PR:edges of size 1 or 2}
	Take an integer $n\geq 3$. Let $q$ be a power of 2, and let $\alpha\in GF(q)^n$ with $\sigma:=\alpha_1+\cdots+\alpha_n\neq 0$. Then the following statements hold:
	\begin{enumerate}[(1)]
		\item the members of size 1 of $\local(S,\alpha)$ are $\{\alpha_1+\sigma\},\ldots,\{\alpha_n+\sigma\}$. 
		\item the members of size 2 of $\local(S,\alpha)$ form a graph that consists of $\frac{q}{2}-1$ connected components $G_1,\ldots,G_{\frac{q}{2}-1}$ satisfying the following: for each $j=1,\ldots,\frac{q}{2}-1$,
		\begin{itemize}
			\item $G_j$'s vertex set is $\left\{\beta_1^j,\beta_1^j+\sigma\right\}\cup\cdots\cup\left\{\beta_n^j,\beta_n^j+\sigma\right\}$ where $\left\{\beta_i^j,\beta_i^j+\sigma\right\}\subseteq U_i- \{\alpha_i+\sigma\}=GF(q)-\{\alpha_i,\alpha_i+\sigma\}$ for $i\in[n]$,
			\item $G_j$ is a bipartite graph with bipartition $\left\{\beta_1^j,\ldots,\beta_n^j\right\}\cup \left\{\beta_1^j+\sigma,\ldots,\beta_n^j+\sigma\right\}$,
			\item $\beta_i^j=\beta_1^j+\alpha_1+\alpha_i$ for $i\in[n]$, and
			\item $G_j$'s edge set is $\left\{\left\{\beta_i^j,\beta_k^j+\sigma\right\}:\;i\neq k\right\}$, i.e., $G_j$ is obtained from a complete bipartite graph after removing the edges of a perfect matching (see Figure~\ref{fig:edges of size 1,2} for an illustration).
		\end{itemize}
	\end{enumerate}
	
\end{LE}
\begin{figure}[h!]
	\begin{center}
		\begin{tikzpicture}
		\draw[fill=black] (-5.3,1) circle (2pt);
		\draw[fill=black] (-4,1) circle (2pt);
		\draw[fill=black] (-2,1) circle (2pt);
		
		\draw[fill=black] (-0.3,0) circle (2pt);
		\draw[fill=black] (1,0) circle (2pt);
		\draw[fill=black] (3,0) circle (2pt);
		\draw[fill=black] (-0.3,2) circle (2pt);
		\draw[fill=black] (1,2) circle (2pt);
		\draw[fill=black] (3,2) circle (2pt);
		
		\node[-] (3) at (2,0) {$\cdots$};
		\node[-] (c) at (2,2) {$\cdots$};
		\node[-,label=below:{$\beta_1^j+\sigma$}] (1) at (-0.3,0) {};
		\node[-,label=above:{$\beta_1^j$}] (a) at (-0.3,2) {};
		\node[-,label=below:{$\beta_2^j+\sigma$}] (2) at (1,0) {};
		\node[-,label=above:{$\beta_2^j$}] (b) at (1,2) {};
		\node[-,label=below:{$\beta_n^j+\sigma$}] (4) at (3,0) {};
		\node[-,label=above:{$\beta_n^j$}] (d) at (3,2) {};
		\node[-,label=below:{$G_j$ for $j=1,\ldots,\frac{q}{2}-1$}] (0) at (1.5,-1) {};
		
		\node[-,label=below:{$\alpha_1+\sigma$}] (5) at (-5.3,1) {};
		\node[-,label=below:{$\alpha_2+\sigma$}] (6) at (-4,1) {};
		\node[-] (7) at (-3,1) {$\cdots$};
		\node[-,label=below:{$\alpha_n+\sigma$}] (8) at (-2,1) {};
		
		\path[thick] (1) edge (b);
		\path[thick] (1) edge (d);
		\path[thick] (2) edge (a);
		\path[thick] (2) edge (d);
		\path[thick] (4) edge (a);
		\path[thick] (4) edge (b);
		\end{tikzpicture}
		\caption{Members of size 1 and 2 of $\local(S,\alpha)$}\label{fig:edges of size 1,2}
	\end{center}
\end{figure}
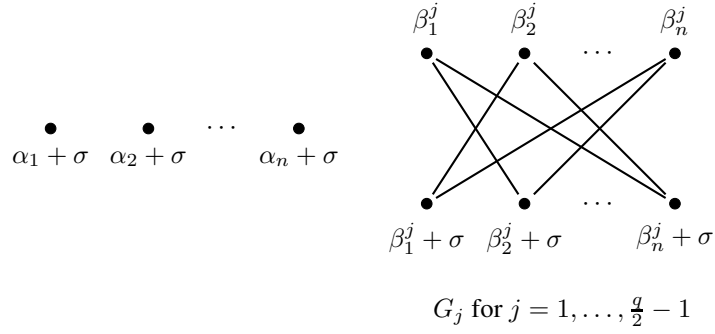 
\begin{proof}
See \S\ref{sec:proof-hyperedge-prop} of the appendix. 
\end{proof}

%% file: q=4.tex
\section{The $q=4$ case}\label{sec:q=4}

In this section, we prove \Cref{q=4} characterizing when the multipartite uniform clutter of a vector space over $GF(4)$ is ideal. The proof of~\Cref{q=4} uses the following two lemmas. 
We first show \Cref{GF4-ideal} which implies that $\mult(T)$ is ideal if $T$ is a vector space over $GF(4)$ such that $\mat(T)\cong\mat(A_n)$ for some $n\geq 3$. We then prove in \Cref{RE:subdivision} that idealness is closed under ``series extensions".

\begin{LE}\label{GF4-ideal}
Let $T=\{x\in GF(4)^{n}:\;x_1+\cdots+x_n=0\}$ for some $n\geq 3$. Then $\mult(T)$ is ideal.
\end{LE}
\begin{proof}
	By \Cref{idealness-local}, it suffices to argue that all localizations of $\mult(T)$ are ideal. Let $\alpha=(\alpha_1,\ldots,\alpha_n)\not\in T$. We will show that the localization of $\mult(T)$ with respect to $\alpha$, denoted $\ind(T,\alpha)$, is ideal. 
	Let $\sigma=\alpha_1+\cdots+\alpha_n\neq0$. Note that $\local(T,\alpha)$ has $n$ members of cardinality~1, $\{\alpha_1+\sigma\},\ldots,\{\alpha_n+\sigma\}$ by \Cref{PR:edges of size 1 or 2}~(1). By \Cref{PR:edges of size 1 or 2}~(2), the members of cardinality~2 form a connected bipartite graph $G$ where
		\begin{itemize}
			\item $G$ is bipartite on $\left\{\beta_1,\ldots,\beta_n\right\}\cup \left\{\beta_1+\sigma,\ldots,\beta_n+\sigma\right\}$ where $\{\beta_i,\beta_i+\sigma\}=GF(4)-\{\alpha_i,\alpha_i+\sigma\}$ for $i\in[n]$,
			\item $\beta_i=\beta_1+\alpha_1+\alpha_i$ for $i\in[n]$, and
			\item the edge set of $G$ is $\left\{\{\beta_i,\beta_k+\sigma\}:\;i\neq k\right\}$.
		\end{itemize}
		We will show that there is no  member of cardinality at least~3 in $\local(T,\alpha)$. Suppose for a contradiction that $\local(T,\alpha)$ has a member $C$ whose cardinality is at least~3. As $C$ does not contain any of the members of $\ind(T,\alpha)$ that have cardinality 1 or 2, $C\subseteq \left\{\beta_1,\ldots,\beta_n\right\}$ or $C\subseteq \left\{\beta_1+\sigma,\ldots,\beta_n+\sigma\right\}$. Without loss of generality, we may assume that $C=\left\{\beta_1,\ldots,\beta_k\right\}$ for some $k\geq 3$. Then, by \Cref{LE:hyperedges}, we have $\sum_{i=1}^k\beta_i=\sigma+\sum_{i=1}^k\alpha_i$. Substituting $\beta_i=\beta_1+\alpha_1+\alpha_i$ for $i=2,\ldots,k$, we obtain $\sum_{i=1}^k\left(\beta_1+\alpha_1\right)=\sigma$. Since $\sigma$ is nonzero and $\sum_{i=1}^k\left(\beta_1+\alpha_1\right)$ is either $\beta_1+\alpha_1$ or $0$, we get $\sum_{i=1}^k\left(\beta_1+\alpha_1\right)=\beta_1+\alpha_1=\sigma$. However, $\beta_1+\alpha_1=\sigma$ in turn implies that $\beta_i=\beta_1+\alpha_1+\alpha_i=\alpha_i+\sigma$, contradicting the assumption that $\beta_i\in GF(4)-\{\alpha_i,\alpha_i+\sigma\}$. Therefore, $\local(T,\alpha)$ does not have a member of cardinality at least~3, as required.
		
		Thus the members of $\local(T,\alpha)$ have size either 1 or 2. Let $\mathcal{C}$ be what is obtained from $\ind(T,\alpha)$ after deleting every element that appears in a member of cardinality~1. As no minimally non-ideal clutter has a member of cardinality~1, $\ind(T,\alpha)$ is ideal if and only if $\mathcal{C}$ is ideal. Notice that $M(\mathcal{C})$, the incidence matrix of $\mathcal{C}$, is the edge - vertex incidence matrix of a bipartite graph. It follows from K\H{o}nig's theorem for bipartite matching that $\mathcal{C}$ is ideal. Therefore, $\ind(T,\alpha)$ is ideal, and $\mult(T)$ is ideal, as required.
	\end{proof}

\begin{LE}\label{RE:subdivision}
Suppose that $S$ is a vector space over $GF(q)$ such that $\mat(S)$ has elements in series. Let $S^\prime$ be a projection of $S$ obtained after dropping one of the elements in series. Then $\mult(S)$ is ideal if and only if $\mult(S^\prime)$ is ideal.
\end{LE}
\begin{proof}
Without loss of generality, assume that $\mat(S)$ has $n$ elements and that elements $n-1, n$ are in series. Let $S^\prime$ be defined as the projection of $S$ obtained after dropping the $n\textsuperscript{th}$ coordinate of the points in $S$. Then $S^\prime$ is a vector space in $GF(q)^{n-1}$, and by \Cref{LE:matroid-minor}, $\mat(T)=\mat(S)/\{n\}$.
	
	Let $x\in S$. Then $\supp(x)$ is the union of some circuits of $\mat(S)$ by \Cref{RE:supports-matroid}. As $n-1,n$ are series elements, a circuit of $\mat(S)$ contains $n-1$ if and only if it contains $n$, implying in turn that $n-1\in \supp(x)$ if and only if $n\in \supp(x)$. Let $v^1,\ldots,v^r$ give rise to a basis of $S$. If $n\in \supp(x)$ for some $x\in S$, then $n\in \supp(v^\ell)$ for some $\ell\in[r]$, and thus, we may assume that $n\in \supp(v^1)$ and that $v^1_n\neq 0$. After scaling the $v^\ell$'s, if necessary, we may assume that $v^\ell_n=0$ for $\ell\in[r]-\{1\}$. Since $n-1\in \supp(x)$ if and only if $n\in \supp(x)$ for $x\in S$, we have that $v^1_{n-1}\neq 0$ and $v^\ell_{n-1}=0$ for $\ell\in[r]-\{1\}$. Then for some $y,z\in GF(q)-\{0\}$,
	\[
	\begin{matrix}
	v^1\\
	v^2\\
	\vdots\\
	v^r
	\end{matrix}
	\left[
	\begin{array}{ccc|cc}
	&\cdots& &y & z\\
	&\cdots& &0 & 0\\
	&\vdots& &0 & 0\\
	&\cdots& &0 & 0
	\end{array}
	\right].
	\]
	Then it follows that $S=\left\{(x_1,\ldots,x_{n-1},zy^{-1}x_{n-1}):\ (x_1,\ldots, x_{n-1})\in S^\prime\right\}$, and by \Cref{RE:relabel-elts}, $\mult(S)\cong \mult(T)$ where
	$T=\left\{(x_1,\ldots,x_{n-1},x_{n-1}):\ (x_1,\ldots, x_{n-1})\in S^\prime\right\}$. Let $V_1\cup\cdots\cup V_n$ be the ground set of $\mult(S)$ where each $V_i$ is a copy of $GF(q)$. Then
	$$\mult(T) = \left\{C:\ C'\in \mult(S^\prime),\ C\cap\left(V_1\cup \cdots \cup V_{n-1}\right)=C^\prime,\ C\cap V_n = C^\prime\cap V_{n-1}\right\}.$$
In words, $\mult(T)$ is obtained from $\mult(S^\prime)$ after duplicating the element in $V_{n-1}$ of each member $C^\prime\in\mult(S^\prime)$. Then the $V_{n-1}$ part and the $V_n$ part of the members of $\mult(T)$ are identical. Hence, $\mult(T)$ is ideal if and only if $\mult(S^\prime)$. As $\mult(S)$ is isomorphic to $\mult(T)$, it follows that $\mult(S)$ is ideal if and only if $\mult(S^\prime)$ is ideal.	
\end{proof}

Now we are ready to prove \Cref{q=4}. The proof first reduces to the case when the vector space $T$ admits a sunflower basis. Then the idea is to show that $\mat(T)$ is a series extension of $\mat(T')$ where $\mat(T')\cong \mat(A_{t})$ for some $t\geq 3$. We then use Lemmas~\ref{GF4-ideal} and~\ref{RE:subdivision} to show that $\mult(T)$ is ideal.

\begin{proof}[{\bf Proof of \Cref{q=4}}]
		Take an integer $n\geq 1$, and let $S\subseteq GF(4)^n$ be a vector space over $GF(4)$. 
First of all, {\bf (i)}$\Rightarrow${\bf (iii)} is straightforward as $\Delta_3$ is non-ideal. In what follows, we will show directions {\bf (iii)}$\Rightarrow${\bf (ii)} and {\bf (ii)}$\Rightarrow${\bf (i)}.

	{\bf (iii)$\Rightarrow$(ii)}: By \Cref{LE:no U24 K4/e},  $\mat(S)=M_1\oplus\cdots\oplus M_k$ for some $k\geq 1$ where for each $i\in[k]$, $M_i$ is the cycle matroid of a bridge, a circuit, or a subdivision $A_t$ for some $t\geq 3$. Then it follows from \Cref{almost-disj-} that $S$ satisfies {\bf (ii)}.
	
	{\bf (ii)$\Rightarrow$(i)}: It suffices to show that $\mult(S_i)$ is ideal for every $i\in[k]$ due to \Cref{set-product}. To this end, take an $i\in[k]$. If $S_i$ has dimension at most 1, then $S_i=\{\mathbf{0}\}$ or $S_i=\langle v\rangle$ for some nonzero vector $v$, in which case it follows from \Cref{disj->ideal} that $S_i$ is ideal.
	Thus we may assume that $S_i=\langle v^1,\ldots, v^r\rangle$ where $r\geq 2$ and $v^1,\ldots, v^r$ give rise to a sunflower basis of $S_i$. 
	Let $T'=\langle w^1,\ldots, w^r\rangle$ where 
	\[
	\begin{matrix}
	w^1\\
	w^2\\
	\vdots\\
	w^r
	\end{matrix}
	\left[
	\begin{array}{c|c|c|c|c}
	1 & 1 & 0 & \cdots & 0\\
	1 & 0 & 1 & \cdots & 0\\
	\vdots    & \vdots    &   \vdots  & \vdots & \vdots\\
	1 & 0 & 0 & \cdots & 1
	\end{array}
	\right].
	\]
	Then $T'=\left\{x\in GF(4)^{r+1}:x_1+\cdots+x_{r+1}=0\right\}$, so by \Cref{GF4-ideal}, $\mult(T')$ is ideal. Suppose that $v^i$ is of the form $(u^0, u^i)$ for $i\in[r]$, and let $d_\ell$ denote the number of entries in $u^\ell$ for $\ell=0,1,\ldots, r$. Then we define $T$ as
	\[
	T:=\left\{(\underbrace{x_1,\ldots,x_1}_{d_0},\underbrace{x_2,\ldots,x_2}_{d_1},\ldots,\underbrace{x_{r+1},\ldots,x_{r+1}}_{d_r}):\;(x_1,x_2,\ldots,x_{r+1})\in T'\right\}.
	\]
	Then $T$ is generated by $y^1,\ldots,y^r$ where 
	\[
	\begin{matrix}
	\\
	y^1\\
	y^2\\
	\vdots\\
	y^r
	\end{matrix}
	\left[
	\begin{array}{c|c|c|c|c}
	\overbrace{\mathbf{1}}^{d_0} & \overbrace{\mathbf{1}}^{d_1} & \overbrace{\mathbf{0}}^{d_2} & \cdots & \overbrace{\mathbf{0}}^{d_r}\\
	\mathbf{1} & \mathbf{0} & \mathbf{1} & \cdots & \mathbf{0}\\
	\vdots    & \vdots    &   \vdots  & \vdots & \vdots\\
	\mathbf{1} & \mathbf{0} & \mathbf{0} & \cdots & \mathbf{1}
	\end{array}
	\right].
	\]
	Note that $T'$ is a projection of $T$ obtained after dropping the coordinates that correspond to some series elements of $\mat(T)$. As $\mult(T')$ is ideal, it follows from \Cref{RE:subdivision} that $\mult(T)$ is ideal. Moreover, $S_i$ can be obtained from $T$ by taking coordinate-wise bijections. Hence, \Cref{RE:relabel-elts} implies that $\mult(S_i)\cong \mult(T)$, thereby showing that $\mult(S_i)$ is ideal, as required.
\end{proof}

%% file: q_4.tex
\section{Powers of 2 greater than 4}\label{sec:q>4}

In this section, we prove \Cref{q>4} which characterizes when the multipartite uniform clutter of a vector space $S$ over $GF(2^k)$ with $k>2$ is ideal. We start by proving Lemmas~\ref{PR:parallel} and~\ref{LE:non-disjoint-supports'} which imply that if $\mult(S)$ is ideal, then the underlying matroid $\mat(S)$ does not contain two distinct circuits that intersect. The proofs of the lemmas rely on the tools from~\Cref{sec:p=2-At}.

For the first lemma, recall that $C_5^2$ is the clutter of edges in a cycle of length 5, and that $C_5^2$ is non-ideal. 

\begin{LE}\label{PR:parallel}
	Let $q$ be a power of 2 greater than 4, and let $S\subseteq GF(q)^3$ be a vector space over $GF(q)$ such that $\mat(S)$ is isomorphic to $\mat(A_3)$. Then $\mult(S)$ has $C_5^2$ as a minor.
\end{LE}
\begin{proof}
	By \Cref{RE:parallel-generators}, we may assume that $S=\{x\in GF(q)^3:x_1+x_2+x_3=0\}$. Let $\alpha=(\alpha_1,\alpha_2,\alpha_3)\not\in S$. We will show that $\ind(S,\alpha)$ has $C_5^2$ as a minor. 
	Let $\sigma=\alpha_1+\alpha_2+\alpha_3$, and we choose $a,b\in GF(q)$ such that $a\in GF(q)-\{\alpha_1,\alpha_1+\sigma\}$ and $b\in GF(q)-\{\alpha_1,\alpha_1+\sigma,a,a+\sigma\}$.
	\begin{claim}
		$a+b+\alpha_1\in GF(q)-\{\alpha_1,\alpha_1+\sigma,a,a+\sigma,b,b+\sigma\}$.
	\end{claim}
	\begin{cproof}
		If $a+b+\alpha_1=\alpha_1$ or $\alpha_1+\sigma$, then $b=a$ or $b=a+\sigma$, contradicting the choice of $b$. If $a+b+\alpha_1=a$ or $a+\sigma$, then $b=\alpha_1$ or $b=\alpha_1+\sigma$, contradicting the choice of $b$. If $a+b+\alpha_1=b$ or $b+\sigma$, then $a=\alpha_1$ or $a=\alpha_1+\sigma$, a contradiction as $a\not\in \{\alpha_1,\alpha_1+\sigma\}$. Therefore, $a+b+\alpha_1\notin \{\alpha_1,\alpha_1+\sigma,a,a+\sigma,b,b+\sigma\}$, as required.
	\end{cproof}
	
	By \Cref{PR:edges of size 1 or 2}~(2), the members of cardinality~2 in $\local(S,\alpha)$ 
	form a graph with $\frac{q}{2}-1$ connected components $G_1,\ldots, G_{\frac{q}{2}-1}$ where the vertex set of $G_j$ is 
	\[\left\{\beta_1^j,\beta_1^j+\sigma\right\}\cup\left\{\beta_2^j,\beta_2^j+\sigma\right\}\cup\left\{\beta_3^j,\beta_3^j+\sigma\right\}\]
	where  $\beta_i^j,\beta_i^j+\sigma\in U_i-\{\alpha_i+\sigma\}$ and $U_i=GF(q)-\{\alpha_i\}$ for $i\in[3]$. Furthermore, $G_1,\ldots,G_{\frac{q}{2}-1}$ are 6-cycles by \Cref{PR:edges of size 1 or 2}~(2) (see Figure~\ref{fig:subgraph} for an illustration). As $\frac{q}{2}-1\geq 3$, without loss of generality, we may assume that $\beta_1^1=a$, $\beta_1^2=b$, and $\beta_1^3=a+b+\alpha_1$, i.e., $G_1,G_2,G_3$ contain $a,b,a+b+\alpha_1\in U_1-\{\alpha_1+\sigma\}$, respectively.
	\begin{claim}
		The following statements hold:
		\begin{enumerate}[(1)]
			\item $\beta_1^1+\sigma=a+\sigma$, $\beta_2^1+\sigma=a+\alpha_1+\alpha_2+\sigma$, and $\beta_3^1=a+\alpha_1+\alpha_3$,
			\item $\beta_2^2=b+\alpha_1+\alpha_2$ and $\beta_2^2+\sigma=b+\alpha_1+\alpha_2+\sigma$, and
			\item $\beta_3^3+\sigma=a+b+\alpha_3+\sigma$.
		\end{enumerate}
	\end{claim} 
	\begin{cproof}
		The claim follows from \Cref{PR:edges of size 1 or 2}~(2).
	\end{cproof}
	
	Now keep elements $\beta_1^1,\beta_1^1+\sigma,\beta_2^1+\sigma,\beta_3^1$ in $G_1$, $\beta_2^2,\beta_2^2+\sigma$ in $G_2$, and $\beta_3^3+\sigma$ in $G_3$ and delete the other elements from $\ind(S,\alpha)$. %
	(see Figure~\ref{fig:subgraph} for an illustration; we keep only the circled elements).
	\begin{figure}[h!]
		\begin{center}
			\begin{tikzpicture}
			\draw[fill=black] (-0.3,0) circle (2pt);
			\draw[fill=black] (1,0) circle (2pt);
			\draw[fill=black] (1,2) circle (2pt);
			\draw[fill=black] (-0.3,2) circle (2pt);
			\draw[fill=black] (2.3,2) circle (2pt);
			\draw[fill=black] (2.3,0) circle (2pt);
			
			\node[-,label=below:{$\beta_3^1+\sigma$}] (3) at (2.3,0) {};
			\node[-,label=above:{$\beta_3^1$}] (c) at (2.3,2) {};
			\draw[black] (2.3,2) circle (0.25cm);
			\node[-,label=below:{$\beta_1^1+\sigma$}] (1) at (-0.3,0) {};
			\draw[black] (-0.3,0) circle (0.25cm);
			\node[-,label=above:{$\beta_1^1$}] (a) at (-0.3,2) {};
			\draw[black] (-0.3,2) circle (0.25cm);
			\node[-,label=above:{$\beta_2^1$}] (b) at (1,2) {};
			\node[-,label=below:{$\beta_2^1+\sigma$}] (2) at (1,0) {};
			\draw[black] (1,0) circle (0.25cm);
			\node[-,label=below:{$G_1$}] (0) at (1,-1) {};
			
			\path[thick] (1) edge (c);
			\path[thick] (2) edge (a);
			\path[thick] (2) edge (c);
			\path[dotted,thick] (3) edge (a);
			\path[dotted,thick] (1) edge (b);
			\path[dotted,thick] (3) edge (b);
			\end{tikzpicture}	
			\quad\quad
			\begin{tikzpicture}
			\draw[fill=black] (-0.3,0) circle (2pt);
			\draw[fill=black] (1,0) circle (2pt);
			\draw[fill=black] (1,2) circle (2pt);
			\draw[fill=black] (-0.3,2) circle (2pt);
			\draw[fill=black] (2.3,2) circle (2pt);
			\draw[fill=black] (2.3,0) circle (2pt);
			
			\node[-,label=below:{$\beta_3^2+\sigma$}] (3) at (2.3,0) {};
			\node[-,label=above:{$\beta_3^2$}] (c) at (2.3,2) {};
			\node[-,label=below:{$\beta_1^2+\sigma$}] (1) at (-0.3,0) {};
			\node[-,label=above:{$\beta_1^2$}] (a) at (-0.3,2) {};
			\node[-,label=above:{$\beta_2^2$}] (b) at (1,2) {};
			\draw[black] (1,2) circle (0.25cm);
			\node[-,label=below:{$\beta_2^2+\sigma$}] (2) at (1,0) {};
			\draw[black] (1,0) circle (0.25cm);
			\node[-,label=below:{$G_2$}] (0) at (1,-1) {};
			
			\path[dotted,thick] (1) edge (c);
			\path[dotted,thick] (2) edge (a);
			\path[dotted,thick] (2) edge (c);
			\path[dotted,thick] (3) edge (a);
			\path[dotted,thick] (1) edge (b);
			\path[dotted,thick] (3) edge (b);
			\end{tikzpicture}
			\quad\quad
			\begin{tikzpicture}
			\draw[fill=black] (-0.3,0) circle (2pt);
			\draw[fill=black] (1,0) circle (2pt);
			\draw[fill=black] (1,2) circle (2pt);
			\draw[fill=black] (-0.3,2) circle (2pt);
			\draw[fill=black] (2.3,2) circle (2pt);
			\draw[fill=black] (2.3,0) circle (2pt);
			
			\node[-,label=below:{$\beta_3^3+\sigma$}] (3) at (2.3,0) {};
			\node[-,label=above:{$\beta_3^3$}] (c) at (2.3,2) {};
			\node[-,label=below:{$\beta_1^3+\sigma$}] (1) at (-0.3,0) {};
			\node[-,label=above:{$\beta_1^3$}] (a) at (-0.3,2) {};
			\node[-,label=above:{$\beta_2^3$}] (b) at (1,2) {};
			\node[-,label=below:{$\beta_2^3+\sigma$}] (2) at (1,0) {};
			\draw[black] (2.3,0) circle (0.25cm);
			\node[-,label=below:{$G_3$}] (0) at (1,-1) {};
			
			\path[dotted,thick] (1) edge (c);
			\path[dotted,thick] (2) edge (a);
			\path[dotted,thick] (2) edge (c);
			\path[dotted,thick] (3) edge (a);
			\path[dotted,thick] (1) edge (b);
			\path[dotted,thick] (3) edge (b);
			\end{tikzpicture}
			\caption{The subgraph of $H_{n,\alpha}$ after deleting the vertices}\label{fig:subgraph}
		\end{center}
	\end{figure}
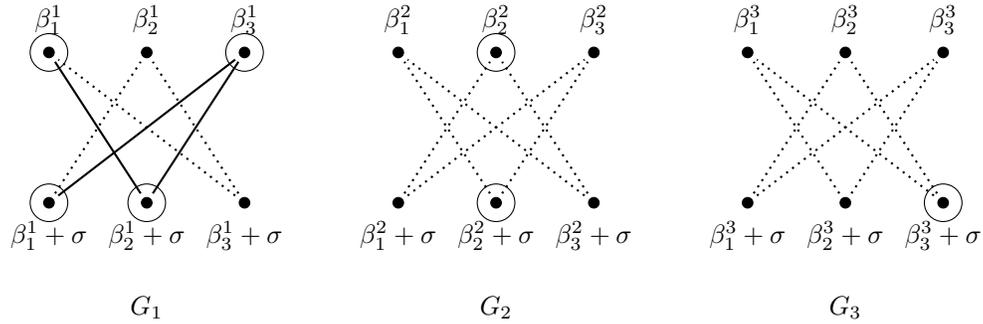
	Let $\mathcal{C}$ denote the resulting minor of $\local(S,\alpha)$.
	
	As $\alpha_i+\sigma$ for $i\in[n]$ are deleted, we know from \Cref{PR:edges of size 1 or 2}~(1) that $\mathcal{C}$ contains no member of size~1. By \Cref{PR:edges of size 1 or 2}~(2), $\mathcal{C}$ has 3 members of size~2: $\left\{\beta_1^1,\beta_2^1+\sigma\right\}$, $\left\{\beta_3^1,\beta_1^1+\sigma\right\}$, $\left\{\beta_3^1,\beta_2^1+\sigma\right\}$, and these are the only ones. (see Figure~\ref{fig:subgraph} for an illustration; the 3 thick edges represent the 3 members of size $2$ in $\mathcal{C}$).
	
	\begin{claim}
		$\left\{\beta_1^1,\beta_2^2,\beta_3^3+\sigma\right\}$ and $\left\{\beta_1^1+\sigma,\beta_2^2+\sigma,\beta_3^3+\sigma\right\}$ are the only members of size greater than 2 in $\mathcal{C}$.
	\end{claim}
	\begin{cproof}
		$\mathcal{C}$ contains at most one element in $U_i$ for $i\in[3]$ by \Cref{LE:hyperedges}, so $\mathcal{C}$ has no member of size greater than~3. Moreover, a member of size 3 contains one element from each $U_1,U_2,U_3$. 
		The subsets of size 3 that do not contain a member of size 2 but one element from each of $U_1,U_2,U_3$ are the following:
		\[
		\left\{\beta_1^1,\beta_2^2,\beta_3^1\right\},\;\left\{\beta_1^1,\beta_2^2+\sigma,\beta_3^1\right\},\;\left\{\beta_1^1,\beta_2^2,\beta_3^3+\sigma\right\},\;\left\{\beta_1^1,\beta_2^2+\sigma,\beta_3^3+\sigma\right\},
		\]
		\[
		\left\{\beta_1^1+\sigma,\beta_2^1+\sigma,\beta_3^3+\sigma\right\},\;\left\{\beta_1^1+\sigma,\beta_2^2,\beta_3^3+\sigma\right\},\;\left\{\beta_1^1+\sigma,\beta_2^2+\sigma,\beta_3^3+\sigma\right\}.
		\]
		By \Cref{LE:hyperedges}, a subset $\{x_1,x_2,x_3\}$ where $x_i\in U_i$ for $i=1,2,3$ is a member if and only if $x_1+x_2+x_3=\sigma+\alpha_1+\alpha_2+\alpha_3$.
		Notice that $\beta_1^1+\beta_2^2+\beta_3^1=b+\alpha_2+\alpha_3$ cannot be $\sigma+\alpha_1+\alpha_2+\alpha_3$, because $b$ is not $\alpha_1+\sigma$ by our choice of $b$. This implies that $\left\{\beta_1^1,\beta_2^2,\beta_3^1\right\}$ is not a member. Similarly, $\left\{\beta_1^1,\beta_2^2+\sigma,\beta_3^1\right\}$ is not a member, because $b\neq \alpha_1$. Notice also that $\{\beta_1^1+\sigma,\beta_2^1+\sigma,\beta_3^3+\sigma\}$ is not a member, because   $\beta_1^1+\sigma+\beta_2^1+\sigma+\beta_3^3+\sigma=a+b+\alpha_1+\alpha_2+\alpha_3+\sigma$ cannot be $\sigma+\alpha_1+\alpha_2+\alpha_3$ by our assumption that $a\neq b$. Observe that $\beta_1^1+\beta_2^2+\beta_3^3+\sigma=\sigma+\alpha_1+\alpha_2+\alpha_3$, implying in turn that $\left\{\beta_1^1,\beta_2^2,\beta_3^3+\sigma\right\}$ and $\left\{\beta_1^1+\sigma,\beta_2^2+\sigma,\beta_3^3+\sigma\right\}$ are members, whereas $\left\{\beta_1^1,\beta_2^2+\sigma,\beta_3^3+\sigma\right\}$ and $\left\{\beta_1^1+\sigma,\beta_2^2,\beta_3^3+\sigma\right\}$ are not. Therefore, $\left\{\beta_1^1,\beta_2^2,\beta_3^3+\sigma\right\}$ and $\left\{\beta_1^1+\sigma,\beta_2^2+\sigma,\beta_3^3+\sigma\right\}$ are the only members of size at least 3 in $\mathcal{C}$, as required.
	\end{cproof}
	
	Now that we have characterized all members of $\mathcal{C}$, we know that the incidence matrix of the corresponding minor $\mathcal{C}$ is the following 0,1 matrix:
	\[
	\bordermatrix{&\beta_1^1 & \beta_2^1+\sigma & \beta_3^1 & \beta_1^1+\sigma & \beta_3^3+\sigma & \beta_2^2 & \beta_2^2+\sigma \cr
		&	 1 & 1 & 0 & 0 & 0 & 0 &0\cr
		&	 0 & 1 & 1 & 0 & 0 & 0 &0\cr
		&	 0 & 0 & 1 & 1 & 0 & 0 &0\cr
		&	 0 & 0 & 0 & 1 & 1 & 0 &1\cr
		&	 1 & 0 & 0 & 0 & 1 & 1 &0
	}
	\]
	Contracting the elements corresponding to $\beta_2^2,\beta_2^2+\sigma$ from $\mathcal{C}$, we obtain a $C_5^2$ minor. Since $\mathcal{C}$ is a minor of $\ind(S,\alpha)$, we deduce that $\ind(S,\alpha)$ also has $C_5^2$ as a minor, as required.
\end{proof}

\begin{LE}\label{LE:non-disjoint-supports'}
Up to isomorphism, $\mat(A_3)$ is the unique minor-minimal matroid with distinct circuits that have a nonempty intersection. Consequently, if two distinct circuits of a matroid intersect, then the matroid has $\mat(A_3)$ as a minor.
\end{LE}
\begin{proof}
Let $M$ be a minor-minimal matroid over ground set $E$ with distinct circuits that intersect. 

Let $C_1,C_2$ be any pair of distinct circuits that intersect. Observe that $C_1\cup C_2=E$, for if not, $M\setminus \overline{C_1\cup C_2}$ would a proper matroid minor with distinct circuits, namely $C_1,C_2$, that intersect, which cannot be the case. Observe further that $I:=C_1\cap C_2$, which by assumption is nonempty, has size one. For if not, for any $e\in I$, $M/(I-\{e\})$ would be a proper matroid minor with distinct circuits, namely $C_1-(I-\{e\}),C_2-(I-\{e\})$, that intersect, which cannot be the case.

In summary, every two circuits that intersect, have $E$ as their union and an intersection of size one. Since $M$ is a matroid, there is a circuit $C_3\subseteq (C_1\cup C_2)-\{e\}$. Clearly, $C_3$ intersects both $C_1,C_2$. Thus, $|C_1\cap C_3|=|C_2\cap C_3|=1$ and $C_1\cup C_3=C_2\cup C_3=E$. It can be readily checked that $|C_1|=|C_2|=2$, implying in turn that $M\cong \mat(A_3)$, as required.
\end{proof}

Now we are ready to prove \Cref{q>4}. The crux of the proof is outlined as follows. If $\mult(S)$ is ideal where $S$ is a vector space over $GF(2^k)$ for some $k>2$, then $\mult(S)$ has no $C_5^2$ as a minor. Then $\mat(S)$ has no two distinct circuits that intersect, by Lemmas~\ref{PR:parallel} and~\ref{LE:non-disjoint-supports'}. Then we use \Cref{disj-supp} to argue that $S$ has a basis with vectors of pairwise disjoint supports. 

\begin{proof}[{\bf Proof of \Cref{q>4}}]
Take an integer $n\geq 1$. Let $q$ be a power of 2 larger than 4, and let $S\subseteq GF(q)^n$ be a vector space over $GF(q)$. 
{\bf (iii)}$\Rightarrow${\bf (ii)}: Since $\mult(S)$ contains no $C_5^2$ as a minor, $\mat(S)$ has no $\mat(A_3)$ as a matroid minor, by \Cref{PR:parallel}. Thus, every two distinct circuits of $\mat(S)$ must be disjoint, by \Cref{LE:non-disjoint-supports'}. This implies that $\mat(S)$ is the cycle matroid of a graph whose blocks are bridges and circuits, so (ii) follows from \Cref{disj-supp}.
{\bf (i)}$\Rightarrow${\bf (iii)} follows immediately from the fact that $C_5^2$ is non-ideal.
{\bf (ii)}$\Rightarrow${\bf (i)} follows immediately from \Cref{disj->ideal}. 
\end{proof}

%% file: replication.tex
\section{The Replication and $\tau=2$ Conjectures}\label{sec:replication}

Let $\mathcal{C}$ be a clutter over ground set $V$. Given the weights of the elements $w\in\mathbb{Z}_+^V$, the minimum weight of a cover of $\mathcal{C}$ can be computed by the following integer linear program:
\[
\tau(\mathcal{C},w)=\min\left\{w^\top x:\;M(\mathcal{C})x\geq \mathbf{1},\;x\in\mathbb{Z}_+^V\right\}
\]
A dual of this integer program is given by the following:
\[
\nu(\mathcal{C},w)=\max\left\{\mathbf{1}^\top y:\;M(\mathcal{C})^\top y\leq w,\;y\in\mathbb{Z}_+^\mathcal{C}\right\},
\]
and this computes the maximum size of a {\it packing} of members of $\mathcal{C}$ such that each element $v$ appears in at most $w_v$ members in the packing. The linear programming relaxations of these two integer programs are the following primal-dual pair:
\[
\begin{array}{lll}
\tau^*(\mathcal{C},w)=&\mbox{minimize} \quad & w^\top x \\
&\mbox{subject to} \quad &M(\mathcal{C})x\geq\mathbf{1}\\
&&x \geq\mathbf{0}
\end{array}
\quad\quad\quad\quad 
\begin{array}{lll}
\nu^*(\mathcal{C},w)=&\mbox{maximize} \quad & \mathbf{1}^\top y \\
&\mbox{subject to} \quad &M(\mathcal{C})^\top y\leq w\\
&&y \geq\mathbf{0}
\end{array}
\]
By linear programming duality, we have that
\[
\tau(\mathcal{C},w)\geq \tau^*(\mathcal{C},w)=\nu^*(\mathcal{C},w)\geq \nu(\mathcal{C},w).\] 
Although $\tau^*(\mathcal{C},w)=\nu^*(\mathcal{C},w)$ always holds, it is not always the case that $\tau(\mathcal{C},w)=\nu(\mathcal{C},w)$. If $\tau(\mathcal{C},w)=\nu(\mathcal{C},w)$ holds for every $w\in\mathbb{Z}_+^V$, we say that $\mathcal{C}$ has the max-flow min-cut property. In fact, the max-flow min-cut property is equivalent to the {\it total dual integrality} for the integer program computing $\tau(\mathcal{C},w)$. Namely, $\mathcal{C}$ has the max-flow min-cut property if and only if the linear system $M(\mathcal{C})x\geq \mathbf{1},\;x\geq\mathbf{0}$ is {\it totally dual integral}. This implies that if $\mathcal{C}$ has the max-flow min-cut property, then $Q(\mathcal{C})$ is integral~\cite{Hoffman74,Edmonds77} and thus $\mathcal{C}$ is ideal. 

As the max-flow min-cut property is a special case of idealness, a natural question is as to when a clutter has the max-flow min-cut property. In this section, we characterize when the multipartite uniform clutter of a vector space over a finite field has the max-flow min-cut property. 

The readers may have already noticed that \Cref{mfmc} is similar to \Cref{q odd} and \Cref{q>4}. As a direct corollary of these theorems, we obtain the following:
\begin{theorem}
Take a prime power $q$ other than $2,4$, and let $S$ be a vector space over $GF(q)$. Then $\mult(S)$ is ideal if and only if $\mult(S)$ has the max-flow min-cut property.
\end{theorem}

Unlike the case when $q\notin\{2,4\}$, there is a vector space over $GF(4)$ whose multipartite uniform clutter is ideal but does not have the max-flow min-cut property. The element set of $GF(4)$ can be represented as $\{0,1,a,b\}$ where $a$ and $b$ are the numbers satisfying the following addition and multiplication tables:
\[
\begin{array}{c|cccc}
+ & 0 & 1 & a & b\\
\hline
0 & 0 & 1 & a & b\\
1 & 1 & 0 & b & a\\
a & a & b & 0 & 1\\
b & b & a & 1 & 0
\end{array}
\quad\quad\quad\quad
\begin{array}{c|cccc}
\times & 0 & 1 & a & b\\
\hline
0 & 0 & 0 & 0 & 0\\
1 & 0 & 1 & a & b\\
a & 0 & a & b & 1\\
b & 0 & b & 1 & a
\end{array}
\]
\begin{EG}\label{EG:GF(4)-ideal}
	{\rm
		Consider $S=\langle (1,1,0),(1,0,1)\rangle\subseteq GF(4)^3$.
		Then 
		\[
		S=\left\{\begin{array}{l}
		(0,0,0),\;(1,1,0),\;(a,a,0),\;(b,b,0),\;(1,0,1),\;(0,1,1),\;(b,a,1),\;(a,b,1),\\
		(a,0,a),\;(b,1,a),\;(0,a,a),\;(1,b,a),\;(b,0,b),\;(a,1,b),\;(1,a,b),\;(0,b,b)
		\end{array}
		\right\}.
		\]
		One can check by using PORTA~\cite{porta} that $\left\{x\in\mathbb{R}_+^{12}:\;M(\mult(S))x\geq \mathbf{1}\right\}$ is an integral polyhedron, so $\mult(S)$ is ideal. Notice further that $\mult(S)$ does not have the max-flow min-cut property, since $S$ contains \[
		\{(0,0,0),(1,1,0),(1,0,1),(0,1,1)\}\cong R_{1,1}\] as a restriction and so $\mult(S)$ has $Q_6$ as a minor by \Cref{projection}.
	}
\end{EG}

We say that clutter $\mathcal{C}$ \emph{packs} if $\tau(\mathcal{C},\mathbf{1})=\nu(\mathcal{C},\mathbf{1})$. We say that $\mathcal{C}$ has the \emph{packing property} if every minor of $\mathcal{C}$ packs. It was observed in \cite{Cornuejols00} that minimally non-ideal clutters do not pack due to Lehman's theorem~\cite{Lehman90} and that if a clutter has the packing property, then it is ideal. Moreover, notice that the packing property is a relaxed notion of the max-flow min-cut property. Here, the \emph{Replication Conjecture} predicts that the packing property implies the max-flow min-cut property. We answer the conjecture in the affirmative for the class of multipartite uniform clutters from coordinate subspaces.

\begin{proof}[{\bf Proof of \Cref{replication}}]
Take a prime power $q$, and let $S$ be a vector space over $GF(q)$. Suppose that $\mult(S)$ has the packing property. Then every minor of $\mult(S)$ packs and is ideal. Note that $\Delta_3$ is non-ideal. Moreover, it is easy to check that $\tau(Q_6,\mathbf{1})=2$ and $\nu(Q_6,\mathbf{1})=1$, which means that $Q_6$ does not pack. Therefore, $\mult(S)$ has none of $\Delta_3$ and $Q_6$ as a minor. Then it follows from~\Cref{mfmc} that $\mult(S)$ has the max-flow min-cut property.
\end{proof}

Next we consider the \emph{$\tau=2$ Conjecture}~\cite{Cornuejols00} which predicts that a stronger statement than the Replication Conjecture holds true. We call a clutter {\it minimally non-packing} if it does not have the packing property but every proper minor of it does. It is known that a minimally non-packing clutter is either ideal or minimally non-ideal~\cite{Cornuejols00}. Here, the $\tau=2$ Conjecture is that if a clutter $\mathcal{C}$ is ideal and minimally non-packing, then its covering number, defined as $\tau(\mathcal{C},\mathbf{1})$, is two. We show that if the multipartite uniform clutter of a coordinate subspace is ideal and minimally non-packing, then its covering number is two.

\begin{proof}[{\bf Proof of \Cref{tau=2}}]
Take a prime power $q$, and let $S$ be a vector space over $GF(q)$. Suppose that $\mult(S)$ is ideal and minimally non-packing. As $\mult(S)$ does not pack, it does not have the max-flow min-cut property. Then by \Cref{mfmc}, $\mult(S)$ has $\Delta_3$ or $Q_6$ as a minor. Note that as $\Delta_3$ is non-ideal but $\mult(S)$ is ideal, $\mult(S)$ has no $\Delta_3$ as a minor. Then it follows that $\mult(S)$ has $Q_6$ as a minor. Since $Q_6$ itself does not pack and every proper minor of $\mult(S)$ packs, $\mult(S)$ is isomorphic to $Q_6$.  In fact, $Q_6$ is ideal and minimally non-packing, and it has covering number two, as required.
\end{proof}

%% file: appendix.tex
\section{Proof of \Cref{PR:graph-K4/e}}\label{sec:graph-K4/e}

We will prove \Cref{PR:graph-K4/e} that characterizes graphs with no $K_4/e$ as a graph minor. Given a graph $G=(V,E)$ and its block decomposition, we may associate $G$ with a bipartite graph $\mathcal{B}(G)$ where
\begin{itemize}
	\item a part of the bipartition of $\mathcal{B}(G)$ consists of the cut-vertices of $G$, 
	\item the other part consists of the blocks of $G$, and
	\item a cut-vertex $u$ and a block $B$ are adjacent in $\mathcal{B}(G)$ if $u$ is a vertex in $B$.
\end{itemize}   
It is well-known that $\mathcal{B}(G)$ is a tree all of whose leaves are blocks of $G$ (see~\cite{Bondy08}). We call a vertex of $G$ that is not a cut vertex an {\it internal vertex}.

\begin{proof}[{\bf Proof of \Cref{PR:graph-K4/e}}]

	Assume that $G$ contains no $K_4/e$ as a graph minor. We will prove by induction on the number of edges that each block of $G$ is a bridge, a circuit, or a subdivision of $A_t$ for some $t\geq3$. The base case is trivial. For the induction step, we may assume that $G$ has at least 3 edges. If $G$ has more than one block, a block of $G$ has less edges than $G$ does, so we may apply the induction hypothesis to each block of $G$. 
	Thus we may assume that $G$ is 2-vertex-connected, in which case, $G$ has no loop.	
	
	Let $e$ be an edge of $G$. By the induction hypothesis, each block of $G-\{e\}$ is a bridge, a circuit, or a subdivision of $A_t$ for some $t\geq 3$. Moreover, since $G$ has no loop, $G-\{e\}$ has no loop either. We first prove the following claim:
	\begin{claim}
	Either $\mathcal{B}(G-\{e\})$ is a single vertex, i.e., $G-\{e\}$ is 2-vertex-connected, or $\mathcal{B}(G-\{e\})$ is a path whose two ends are blocks of $G$ and $e$ is incident to internal vertices of the two end blocks of the path.
	\end{claim}
	\begin{cproof}
	We may assume that $G-\{e\}$ has at least two blocks. Since $G$ is 2-vertex-connected, $e$ connects two distinct blocks $B_1,B_2$ of $G-\{e\}$. Recall that $\mathcal{B}(G-\{e\})$ is a tree, so there is a unique path between $B_1$ and $B_2$ in $\mathcal{B}(G-\{e\})$. Then, after putting $e$ back, the blocks of $G-\{e\}$ on the path between $B_1$ and $B_2$ become a single block in $G$. In fact, since $G$ is 2-vertex-connected, $G$ has no other block. This implies that $G-\{e\}$ has no block other than the ones on $C$. So, $\mathcal{B}(G-\{e\})$ contains no vertex outside $C$, and therefore, $\mathcal{B}(G-\{e\})$ is a path where $B_1,B_2$ are its two ends. If $e$ is not incident to an internal vertex of $B_1$, then $e$ is incident to the cut-vertex of $B_1$, implying that $B_1$ is separated from $B_2$ in $G$, a contradiction. Thus $e$ is incident to an internal vertex of $B_1$. Similarly, $e$ is incident to an internal vertex of $B_2$, as required.  
	\end{cproof}
	
	Next, we claim the following:
	\begin{claim}
	All but at most one block of $G-\{e\}$ are bridges.
	\end{claim}
	\begin{cproof}
	We may assume that $G-\{e\}$ has at least two blocks. Then, by Claim~1, $\mathcal{B}(G-\{e\})$ is a path $B_1,u_1,B_2,\ldots,u_{k-1},B_k$ for some $k\geq 2$, where $B_1,\ldots, B_k$ are the blocks of $G-\{e\}$ and $u_\ell$ is the cut-vertex separating $B_\ell$ and $B_{\ell+1}$ for $\ell\in[k-1]$. Moreover, by Claim~1, $e=u_0u_k$, where $u_0$ is an internal vertex of $B_1$ and $u_k$ is an internal vertex of $B_k$.
		
		Suppose for a contradiction that $G-\{e\}$ has two blocks that are not bridges. Then $B_i, B_j$ for some distinct $i,j\in[k]$ are not bridges. In particular, $B_i$ and $B_j$ have cycles $C_i$ and $C_j$, respectively. Here, both $C_i$ and $C_j$ have at least two edges as $G-\{e\}$ has no loop. After contracting the edges of $B_\ell$ for $\ell\in [k]-\{i,j\}$ from $G-\{e\}$, the vertices in $B_1,\ldots, B_{i-1}$ are identified with $u_{i-1}$, the vertices in $B_{i+1},\ldots, B_{j-1}$ are identified with $u_{j-1}$, and the vertices in $B_{j+1},\ldots, B_{k}$ are identified with $u_{j}$. Therefore, the resulting graph is $u_{i-1},B_i,u_{j-1},B_j,u_j$, where $u_{i-1}$ and $u_j$ are internal vertices of $B_i$ and $B_j$, respectively, and $u_{j-1}$ is the cut-vertex separating $B_i,B_j$. Notice that $e$ connects $u_{i-1}$ and $u_j$ after the contraction, because $u_0,u_k$ were identified with $u_{i-1},u_j$, respectively (see Figure~\ref{fig:K_4/e case 3} for an illustration). 
		\begin{figure}[h!]
			\begin{center}
				\begin{tikzpicture}
				\draw[fill=black] (0,0) circle (3pt);
				\draw[fill=black] (2,0) circle (3pt);
				\draw[fill=black] (4,0) circle (3pt);
				
				\node[-,label=left:{$u_{i-1}$}] (1) at (0,0) {};
				\node[-,label=right:{$u_{j}$}] (3) at (4,0) {};
				\node[-] (2) at (2,0) {};
				\node[-,label=below:{$u_{j-1}$}] (4) at (2,-0.2) {};
				\node[-,label={$e$}] (5) at (2,0.75) {};
				\path[thick, bend left=50] (1) edge (2);
				\path[thick, bend right=50] (1) edge (2);
				\path[thick, bend left=50] (2) edge (3);
				\path[thick, bend right=50] (2) edge (3);
				
				\path[thick, bend left=100] (1) edge (3);
				
				\node[-] (a) at (1,0) {$C_i$};
				\node[-] (a) at (3,0) {$C_j$};
				\end{tikzpicture}
				\caption{$e=u_{i-1}u_j$}\label{fig:K_4/e case 3}
			\end{center}
		\end{figure}
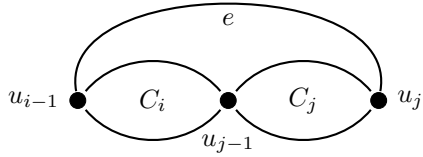
	We then delete the edges outside of the cycles $C_i,C_j$. After adding $e$ back, we obtain a subdivision of $K_4/e$, a contradiction as $G$ has no $K_4/e$ as a graph minor.	Therefore, at most one block of $G-\{e\}$ is a bridge.
	\end{cproof}
	
	If every block of $G-\{e\}$ is a bridge, then it follows from Claim~1 that $G$ is a circuit. Thus we may assume that a block $B$ of $G-\{e\}$ is a circuit or a subdivision of $A_t$ for some $t\geq 3$. Then, by Claim~2, the other blocks of $G-\{e\}$ are bridges. 
	\begin{claim}
		$G$ is the union of $B$ and a path $P$ whose ends are two vertices in $B$ and the other vertices are disjoint from $V(B)$.
	\end{claim}
	\begin{cproof}
		It follows from Claim~1 that $e$ and the bridges of $G-\{e\}$ form a path $P$ connecting two vertices of $B$. An interior vertex of $P$, if exists, is in a block of $G-\{e\}$ other than $B$, so it is not contained in $V(B)$, as required.
	\end{cproof}
	
	As $B$ is a circuit or a subdivision of $A_t$ for some $t\geq 3$, $B$ is a disjoint union of internally vertex-disjoint $uv$-paths for some distinct $u,v\in V(B)$. Let $P_1,\ldots,P_t$ be the $uv$-paths. 
	\begin{claim}
		If $t=2$, $G$ is a subdivision of $A_3$.
	\end{claim}
	\begin{cproof}
		If $t=2$, $B$ is a circuit and $P$ connects two vertices on the cycle by Claim~3. So, $G$ is the union of three internally vertex-disjoint paths connecting the two vertice, implying in turn that $G$ is a subdivision of $A_3$.
	\end{cproof}
	
	By Claim~4, we may assume that $t\geq 3$. We will show that $P$ is also a path connecting $u$ and $v$, thereby proving that $G$ is a subdivision of $A_{t+1}$, obtained from $uv$-paths $P_1,\ldots,P_t,P$.
	\begin{claim}
		$P$ is an $uv$-path.
	\end{claim}
	\begin{cproof}
		Suppose for a contradiction that $P$ is not a $uv$-path. Then one of $P$'s two ends is not in $\{u,v\}$.
		
		First, consider the case when one end of $P$ is in $\{u,v\}$. Without loss of generality, we may assume that one end of $P$ is $u$ and the other end is $w\in V-\{u,v\}$. 
		Without loss of generality, assume that $w$ is on $P_1$. Then the subgraph of $G$ obtained after deleting the edges $E-E(P)\cup E(P_1)\cup E(P_2)\cup E(P_3)$ (see Figure~\ref{fig:K_4/e case 1} for an illustration) is a subdivision of $K_4/e$, contradicting the assumption that $G$ has no $K_4/e$ as a graph minor.
		\begin{figure}[h!]
			\begin{center}
				\begin{tikzpicture}
				\draw[fill=black] (0,0) circle (3pt);
				\draw[fill=black] (4,0) circle (3pt);
				\draw[fill=black] (2.5,0.57) circle (3pt);
				
				\node[-,label=left:{$u$}] (1) at (0,0) {};
				\node[-,label=right:{$v$}] (2) at (4,0) {};
				\node[-,label=above right:{$w$}] (5) at (2.5,0.5) {};
				\node[-,label=above:{$P$}] (6) at (1,0.9) {};
				
				\path[thick, bend left=100] (1) edge (5);
				\path[thick, bend left] (1) edge (2);
				\path[thick, bend right] (1) edge (2);
				\path[thick] (1) edge (2);
				
				\node[-,label=below:{$P_1$}] (a) at (2,0.8) {};
				\node[-,label=below:{$P_2$}] (b) at (2,0.2) {};
				\node[-,label=below:{$P_3$}] (a) at (2,-0.5) {};
				\end{tikzpicture}
				\caption{$w\not\in\{u,v\}$}\label{fig:K_4/e case 1}
			\end{center}
		\end{figure}
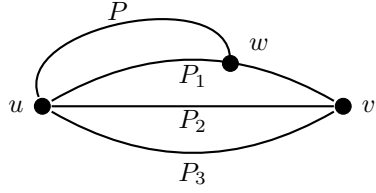
		
		Now consider the case when both ends of $P$ are not in $\{u,v\}$. Let the ends of $P$ be $w_1,w_2\in V-\{u,v\}$. There are two cases to consider: $w_1,w_2$ are on the same $uv$-path of $B$, or $w_1,w_2$ are on different $uv$-paths.
		If $w_1,w_2$ are on the same $uv$-path, we may assume that they are on $P_1$ without loss of generality. In this case, deleting the edges $E-E(P)\cup E(P_1)\cup E(P_2)\cup E(P_3)$ and contracting the edges of the $uw_1$-path on $P_1$ (see Figure~\ref{fig:K_4/e case 2} for an illustration), we obtain a subdivision of $K_4/e$, a contradiction.
		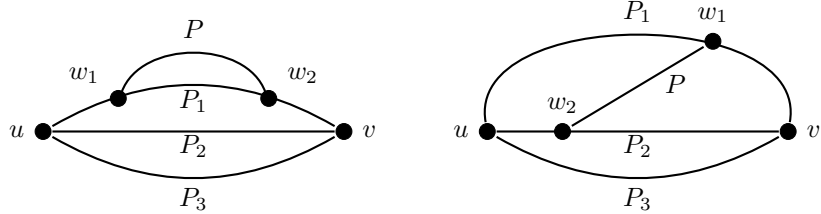
\begin{figure}[h!]
			\begin{center}
				\begin{tikzpicture}
				\draw[fill=black] (0,0) circle (3pt);				\draw[fill=black] (4,0) circle (3pt);
				\draw[fill=black] (3,0.44) circle (3pt);
				\draw[fill=black] (1,0.44) circle (3pt);
				
				\node[-,label=left:{$u$}] (1) at (0,0) {};
				\node[-,label=right:{$v$}] (2) at (4,0) {};
				\node[-,label=above left:{$w_1$}] (7) at (1,0.4) {};
				\node[-,label=above right:{$w_2$}] (5) at (3,0.4) {};
				\node[-,label=above:{$P$}] (6) at (2,1) {};
				
				\path[thick, bend left=70] (7) edge (5);
				\path[thick, bend left] (1) edge (2);
				\path[thick, bend right] (1) edge (2);
				\path[thick] (1) edge (2);
				
				\node[-,label=below:{$P_1$}] (a) at (2,0.8) {};
				\node[-,label=below:{$P_2$}] (b) at (2,0.2) {};
				\node[-,label=below:{$P_3$}] (a) at (2,-0.5) {};
				\end{tikzpicture}
				\quad\quad
				\begin{tikzpicture}
				\draw[fill=black] (0,0) circle (3pt);
				\draw[fill=black] (4,0) circle (3pt);
				\draw[fill=black] (3,1.2) circle (3pt);
				\draw[fill=black] (1,0) circle (3pt);
				
				\node[-,label=left:{$u$}] (1) at (0,0) {};
				\node[-,label=right:{$v$}] (2) at (4,0) {};
				\node[-,label=above:{$w_2$}] (7) at (1,0) {};
				\node[-,label=above:{$w_1$}] (5) at (3,1.2) {};
				\node[-,label=below:{$P$}] (6) at (2.5,1) {};
				
				\path[thick] (7) edge (5);
				\path[thick, bend left=100] (1) edge (2);
				\path[thick, bend right] (1) edge (2);
				\path[thick] (1) edge (2);
				
				\node[-,label=above:{$P_1$}] (a) at (2,1.2) {};
				\node[-,label=below:{$P_2$}] (b) at (2,0.2) {};
				\node[-,label=below:{$P_3$}] (a) at (2,-0.5) {};
				\end{tikzpicture}
				\caption{$w_1,w_2\notin\{u,v\}$}\label{fig:K_4/e case 2}
			\end{center}
		\end{figure}
		
		If $w_1,w_2$ are on different $uv$-paths, we may assume that $w_1$ is on $P_1$ and $w_2$ is on $P_2$ without loss of generality. Deleting the edges $E-E(P)\cup E(P_1)\cup E(P_2)\cup E(P_3)$ and contracting the edges of $P$ (see Figure~\ref{fig:K_4/e case 2} for an illustration), we obtain a subdivision of $K_4/e$, a contradiction as $G$ has no $K_4/e$ as a graph minor. 
	\end{cproof}
	\noindent
	By Claims~3 and~5, $P$ is an $uv$-path that is internally vertex-disjoint from $P_1,\ldots, P_t$, implying in turn that $G$ is a subdivision of $A_{t+1}$. This finishes the proof.
\end{proof}

\section{Proof of \Cref{PR:edges of size 1 or 2}}\label{sec:proof-hyperedge-prop}

\begin{proof}[{\bf Proof of \Cref{PR:edges of size 1 or 2}}]
	{\bf (1)} By \Cref{LE:hyperedges}, $C$ is a member of size 1 if and only if $C=\{\sigma+\alpha_i\}$ for some $i\in[n]$. Therefore, $\{\alpha_1+\sigma\},\ldots,\{\alpha_n+\sigma\}$ are the members of size 1 in $\local(S,\alpha)$, as required.
	
	{\bf (2)} First, we will argue that a member of cardinality 2 contains none of $\alpha_1+\sigma,\ldots,\alpha_n+\sigma$. Let $\{u,v\}$ be a member of size 2 where $u\in U_i$ and $v\in U_j$ for some $i\neq j$. Then we get $u+v=\sigma+\alpha_i+\alpha_j$ by \Cref{LE:hyperedges}. If $u=\alpha_i+\sigma$, then $v=\alpha_j$, contradicting the assumption that $v\in U_j=GF(q)-\{\alpha_j\}$. Therefore, the members of cardinality 2 are contained in $U^\prime:=\left(U_1-\{\alpha_1+\sigma\}\right)\cup\cdots\cup \left(U_n-\{\alpha_n+\sigma\}\right)$. Notice that we have preserved the symmetry between $U_1-\{\alpha_1+\sigma\},\ldots,U_n-\{\alpha_n+\sigma\}$ and that $U_1-\{\alpha_1+\sigma\}$ is not different from the other $U_i-\{\alpha_i+\sigma\}$'s.
	
	Observe that $U_1-\{\alpha_1+\sigma\}=GF(q)-\{\alpha_1,\alpha_1+\sigma\}$ has $q-2$ elements and that $U_1-\{\alpha_1+\sigma\}$ can be partitioned as $U_1-\{\alpha_1+\sigma\}=\left\{\beta_1^1,\beta_1^1+\sigma\right\}\cup\cdots\cup \left\{\beta_1^{\frac{q}{2}-1},\beta_1^{\frac{q}{2}-1}+\sigma\right\}$, with $\frac{q}{2}-1$ sets of cardinality 2, where $\beta_1^1,\ldots,\beta_1^{\frac{q}{2}-1}$ are distinct elements. For $i=2,\ldots, n$ and $j=1,\ldots,\frac{q}{2}-1$, we denote by $\beta_i^j\in U_i$ the element satisfying $\beta_i^j=\beta_1^j+\alpha_1+\alpha_i$. 
	\begin{claim}
		$U_i-\{\alpha_i+\sigma\}=\left\{\beta_i^1,\beta_i^1+\sigma\right\}\cup\cdots\cup \left\{\beta_i^{\frac{q}{2}-1},\beta_i^{\frac{q}{2}-1}+\sigma\right\}$ for $i=1,\ldots,n$.
	\end{claim}
	\begin{cproof}
		We may assume that $i\geq 2$. Let $j,\ell$ be distinct indices in $\left[\frac{q}{2}-1\right]$. As $\beta_1^j\neq\beta_1^\ell$, we get $\beta_i^j\neq \beta_i^\ell$. Similarly, $\beta_1^j\neq\beta_1^\ell+\sigma$ implies $\beta_i^j\neq \beta_i^\ell+\sigma$. Therefore, $\beta_i^1,\beta_i^1+\sigma,\ldots,\beta_i^{\frac{q}{2}-1},\beta_i^{\frac{q}{2}-1}+\sigma$ are distinct elements, so $\left\{\beta_i^1,\beta_i^1+\sigma\right\},\cdots,\left\{\beta_i^{\frac{q}{2}-1},\beta_i^{\frac{q}{2}-1}+\sigma\right\}$ partition $U_i-\{\alpha_i+\sigma\}$, as required.
	\end{cproof}
	
	By Claim~1, each element in $U^\prime$ is $\beta_i^j$ or $\beta_i^j+\sigma$ for some $i\in[n]$ and $j\in\left[\frac{q}{2}-1\right]$. Now we are ready to characterize what the members of size 2 are.
	\begin{claim}
		Let $u,v$ be distinct elements in $U^\prime$. Then $\{u,v\}$ is a member in $\local(S,\alpha)$ if and only if for some $j\in\left[\frac{q}{2}-1\right]$ and distinct $i,k\in[n]$, we have $u=\beta_i^j$ and $v=\beta_k^j+\sigma$ or $u=\beta_i^j+\sigma$ and $v=\beta_k^j$ . 
	\end{claim}
	\begin{cproof}
		{\bf ($\Leftarrow$)} Without loss of generality, we may assume that $j=1$, $i=1$, and $k=2$.  As $\beta_2^1=\beta_1^1+\alpha_1+\alpha_2$, we have $\beta_1^1+\beta_2^1+\sigma=\alpha_1+\alpha_2+\sigma$. So, by \Cref{LE:hyperedges}, $\{u,v\}$ is a member.
		
		{\bf ($\Rightarrow$)} Without loss of generality, we may assume that $u\in U_1$, $v\in U_2$. Then $u=\beta_1^j$ or $u=\beta_1^j+\sigma$ for some $j\in\left[\frac{q}{2}-1\right]$. If $u=\beta_1^j$, then by \Cref{LE:hyperedges}, $v=\beta_1^j+\alpha_1+\alpha_2+\sigma=\beta_2^j+\sigma$. Similarly, if $u=\beta_1^j+\sigma$, we can argue that $v=\beta_2^j$, as required.
	\end{cproof}
	
	For $j\in\left[\frac{q}{2}-1\right]$, let $G_j$ denote the graph induced by the elements in $\left\{\beta_1^j,\ldots,\beta_n^j\right\}\cup \left\{\beta_1^j+\sigma,\ldots,\beta_n^j+\sigma\right\}$. By Claim~2, the edge set of $G_j$ is precisely $\left\{\left\{\beta_i^j,\beta_k^j+\sigma\right\}:\;i\neq k\right\}$. Moreover, Claim~2 also implies that there is no edge between $G_j$ and $G_\ell$ if $j\neq \ell$, as required.
\end{proof}

%% file: main.bbl
\begin{thebibliography}{99}

\bibitem{Abdi-pm-resistant}
Abdi, A. and Cornu\'{e}jols, G.:
The max-flow min-cut property and $\pm$1-resistant sets.
Discrete Applied Mathematics, {\bf 289}, 455--476
(2020)

\bibitem{Abdi-2esistant}
Abdi, A. and Cornu\'{e}jols, G.:
Idealness and 2-resistant sets.
Operations Research Letters, {\bf 47}(5), 358--362
(2019)

\bibitem{Abdi-resistant}
Abdi, A., Cornu\'{e}jols, G., Lee, D.:
Resistant sets in the unit hypercube.
Math. Oper. Res. {\bf 46}(1), 82--114
(2020)

\bibitem{Abdi-cuboid}
Abdi, A., Cornu\'{e}jols, G., Guri\u{c}anov\'{a}, N., Lee, D.:
Cuboids, a class of clutters.
J. Combin. Theory Ser. B {\bf 142}, 144--209
(2020)

\bibitem{Abdi-idealmnp}
Abdi, A., Cornu\'{e}jols, G., Pashkovich, K.:
Ideal clutters that do not pack.
Math. Oper. Res. {\bf 43}(2), 533--553 
(2018)

	
\bibitem{Berge72}
Berge, C.:
Balanced matrices.
Math. Program. {\bf 2}(1), 19--31
(1972)

\bibitem{Bondy08}
Bondy, J.A. and Murty, U.S.R.:
Graph Theory. Springer
(2008)

\bibitem{Brylawski71}
Brylawski, T.H.:
A combinatorial model for series-parallel networks.
Trans. Amer. Math. Soc. {\bf 154}, 1--22
(1971)

\bibitem{Conforti93}
Conforti, M. and Cornu\'{e}jols, G.:
Clutters that pack and the max-flow min-cut property: a conjecture.
(Available online at \texttt{http://www.dtic.mil/dtic/tr/fulltext/u2/a277340.pdf})
The Fourth Bellairs Workshop on Combinatorial Optimization 
(1993)

\bibitem{Cornuejols01}
  Cornu\'{e}jols, G.:
  Combinatorial Optimization, Packing and Covering.
  SIAM, Philadelphia 
 (2001)
 
\bibitem{Cornuejols00}
Cornu\'{e}jols, G., Guenin, B., Margot, F.:
The packing property.
Math. Program. {\bf 89}(1), 113--126
(2000)

\bibitem{Cornuejols94}
Cornu\'{e}jols, G. and Novick, B.:
Ideal 0,1 matrices.
J. Combin. Theory Ser. B {\bf 60}, 145--157
(1994)

\bibitem{Ding08}
Ding, G., Feng, L., Zang, W.:
The complexity of recognizing linear systems with certain integrality properties.
Math. Program. {\bf 114}, 321--334
(2008) 

 \bibitem{Duffin62}
  Duffin, R.J.:
  The extremal length of a network.
  J. Math. Analysis and Appl. {\bf 5}(2), 200--215
  (1962)
	
\bibitem{Edmonds70}
Edmonds, J. and Fulkerson, D.R.:
Bottleneck extrema.
J. Combin. Theory Ser. B {\bf 8}, 299--306
(1970)

\bibitem{Edmonds77}
Edmonds, J. and Giles, R.:
A min-max relation for submodular functions on graphs.
Ann. Discrete Math. {\bf 1}, 185--204
(1977)   

\bibitem{Edmonds73}
Edmonds, J. and Johnson, E.L.:
Matchings, Euler tours and the Chinese postman problem.
Math. Program. {\bf 5}, 88--124
(1973)

\bibitem{Guenin01}
Guenin, B.:
A characterization of weakly bipartite graphs.
J. Combin. Theory Ser. B {\bf 83}, 112--168
(2001)

\bibitem{Hoffman74}
Hoffman, A.J.: 
A generalization of max flow-min cut. 
Math. Program. {\bf 6}(1), 352--359 
(1974)  

\bibitem{Hoffman56}
Hoffman, A.J. and Kruskal J.B.:
Integral boundary points of convex polyhedra. 
In {\it Linear inequalities and related systems} (eds. Kuhn H.W. and Tucker A.W.).
Ann. Math. Studies {\bf 38}, 223--246
(1956)

\bibitem{Lehman79}
Lehman, A.:
On the width-length inequality.
Math. Program. {\bf 17}(1), 403--417
(1979)

\bibitem{Lehman90}
Lehman, A.:
The width-length inequality and degenerate projective planes.
DIMACS Vol. {\bf 1}, 101--105
(1990)

\bibitem{Lovasz72a}
Lov\'{a}sz, L.:
Minimax theorems for hypergraphs.
Lecture Notes in Mathematics. {\bf 411}, Springer-Verlag 111--126
(1972)

\bibitem{Lovasz72}
Lov\'{a}sz, L.:
Normal Hypergraphs and the Perfect Graph Conjecture.
Discrete Math. {\bf 2}, 253--267
(1972)

\bibitem{Lucchesi78}
Lucchesi, C.L. and Younger, D.H.:
A minimax relation for directed graphs.
J. London Math. Soc. {\bf 17} (2), 369--374
(1978)

\bibitem{Menger27}
Menger, K.:
Zur allgemeinen Kurventheorie.
Fundamenta Mathematicae {\bf 10}, 96--115
(1927)

\bibitem{Oxley11}
Oxley, J.:
Matroid Theory, second edition.
Oxford University Press, New York
(2011)

\bibitem{porta}
Christof, T. and L\"{o}bel, A.:
PORTA - A Polyhedron Representation and Transformation
Algorithm, http://porta.zib.de/.

  \bibitem{Seymour76b}
  Seymour, P.D.:
  A forbidden minor characterization of matroid ports.
  Quart. J. Math. {\bf 27}(4), 407--413
  (1976)

\bibitem{Seymour76}
Seymour, P.D.: 
The forbidden minors of binary clutters. 
J. London Math. Society {\bf 2}(12), 356--360 
(1976)

\bibitem{Seymour81}
Seymour, P.D.:
Matroids and multicommodity flows.
Europ. J. Combinatorics {\bf 2}, 257--290
(1981)

\bibitem{Seymour79}
Seymour, P.D.: 
Sums of circuits. 
Graph Theory and Related Topics (Bondy, J.A. and Murty, U.S.R., eds), Academic Press, New York,
342--355
(1979)

\bibitem{Seymour77}
Seymour, P.D.:
The matroids with the max-flow min-cut property.
J. Combin. Theory Ser. B {\bf 23}, 189--222
(1977)

\end{thebibliography}
